\documentclass[11pt]{amsart}
\usepackage{tikzit}

\tikzstyle{Boxes}=[fill=none, draw=black, shape=rectangle, tikzit shape=rectangle, tikzit draw=black]

\tikzstyle{Arrow}=[->]

\usepackage{amssymb,latexsym}

\usepackage{graphicx}
\usepackage{enumerate}
\usepackage{enumitem}
\usepackage{comment}
\usepackage{hyperref}
\usepackage{mathrsfs}

\makeatletter

\@namedef{subjclassname@2010}{

  \textup{2010} Mathematics Subject Classification}
\usepackage{amssymb, amsmath,amsthm}
\newtheorem{thm}{Theorem}[section]
\newtheorem{prop}[thm]{Proposition}
\newtheorem{conj}[thm]{Conjecture}

\newtheorem{cor}[thm]{Corollary}
\newtheorem{lem}[thm]{Lemma}
\theoremstyle{definition}
\newtheorem{rem}[thm]{Remark}
\newtheorem{def1}[thm]{Definition}

\newcommand{\bk}{\backslash}
\newcommand{\mc}{\mathcal}
\newcommand{\mb}{\mathbb}

\renewcommand{\ss}{\substack}

\renewcommand{\deg}[1]{\textnormal{deg}(#1)}

\newcommand{\La}{\langle}
\newcommand{\Ra}{\rangle}
\newcommand{\llf}{\left\lfloor}
\newcommand{\lla}{\left\langle}
\newcommand{\e}{\varepsilon}

\newcommand{\rrf}{\right\rfloor}
\newcommand{\rra}{\right\rangle}
\newcommand{\mbf}{\boldsymbol}
\newcommand{\ra}{\to}
\newcommand{\asum}{\sideset{}{^{\ast}}\sum}

\renewcommand{\bar}{\overline}

\renewcommand{\Re}{\textnormal{Re}}
\let\oldpmod\pmod
\renewcommand{\pmod}[1]{\hspace{-0.1cm}\oldpmod {#1}}

\begin{document}
\title[Erd\H{o}s discrepancy in function fields]{Beyond the Erd\H{o}s discrepancy problem in function fields}

\author
{Oleksiy Klurman}
\address{School of Mathematics,
University of Bristol, Woodland Road, Bristol, BS8 1UG, UK}
\email{lklurman@gmail.com}

\author{Alexander P. Mangerel}
\address{
Department of Mathematical Sciences, Durham University, Upper Mountjoy Campus, Stockton Road, Durham, UK}
\email{smangerel@gmail.com}

\author{Joni Ter\"av\"ainen}
\address{Department of Mathematics and Statistics \\
University of Turku, 20014 Turku\\
Finland}
\email{joni.p.teravainen@gmail.com}

\begin{abstract}
We characterize the limiting behavior of partial sums of multiplicative functions $f:\mb{F}_q[t]\to S^1$. In contrast to the number field setting, the characterization depends crucially on whether the notion of discrepancy is defined using \emph{long intervals}, \emph{short intervals}, or \emph{lexicographic intervals}.

Concerning the notion of short interval discrepancy, we show that a completely multiplicative $f:\mathbb{F}_q[t]\to\{-1,+1\}$ with $q$ odd has bounded short interval sums if and only if $f$ coincides with a ``modified" Dirichlet character to a prime power modulus. This confirms the function field version of a conjecture over $\mathbb{Z}$ that such modified characters are extremal with respect to partial sums. 

Regarding the lexicographic discrepancy, we prove that the discrepancy of a completely multiplicative sequence is always infinite if we define it using a natural lexicographic ordering of $\mathbb{F}_{q}[t]$. This  answers a question of Liu and Wooley.

Concerning the long sum discrepancy, it was observed by the Polymath 5 collaboration that the  Erd\H{o}s discrepancy problem admits infinitely many completely multiplicative counterexamples on $\mathbb{F}_q[t]$. Nevertheless, we are able to classify the  counterexamples if we restrict to the class of modified Dirichlet characters. In this setting, we determine the precise growth rate of the discrepancy, which is still unknown for the analogous problem over the integers.  
\end{abstract}

\subjclass[2020]{11T55, 11K38, 11N37}

\maketitle

\section{Introduction and Results}

The Erd\H{o}s Discrepancy Problem (EDP), formulated in~\cite{erdos1957} (see also~\cite{Katsur},~\cite{Chudak} for related questions), states that, given any sequence $f : \mb{N} \to \{-1,+1\}$, the \emph{discrepancy} of $f$ on homogeneous arithmetic progressions satisfies
\begin{align}\label{eq1}
\sup_{d, N \geq 1} \Big|\sum_{n \leq N} f(dn)\Big| = \infty.
\end{align}
This was eventually settled affirmatively in a groundbreaking paper of Tao~\cite{TaoEDP} in 2015.

The special case where $f$ is \emph{completely multiplicative} (that is, $f(mn)=f(m)f(n)$ for all $m,n\in \mathbb{N}$) was already highlighted by Erd\H{o}s as the key special case; in this case, the formulation simplifies to  
\begin{align}\label{eq2}
\sup_{N \geq 1} \Big|\sum_{n \leq N} f(n)\Big| = \infty,\quad f:\mathbb{N}\to \{-1,+1\}\,\,\textnormal{completely multiplicative}.
\end{align}
The Polymath 5 online collaboration project~\cite{polymath5} devoted to the Erd\H{o}s discrepancy problem was indeed able to reduce the proof of~\eqref{eq1} to (an averaged version of) the completely multiplicative case~\eqref{eq2}, with $f$ now taking values on the unit circle $S^1:=\{z\in \mathbb{C}:\,\, |z|=1\}$ of the complex plane. Tao established in~\cite{TaoEDP} this case of completely multiplicative functions, and hence the whole conjecture~\eqref{eq1}, making crucial use of his proof~\cite{tao} of the logarithmic two-point Elliott conjecture on correlations of multiplicative functions. A further reason to concentrate on the discrepancy of completely multiplicative sequences is that such sequences or small perturbations thereof are speculated  to have minimal growth rate for the discrepancy among all sequences, as discussed below.
 
In this paper we shall consider corresponding discrepancy problems in function fields. Let $q$ be a fixed prime power and let $\mc{M}$ denote the set of monic polynomials in $\mb{F}_q[t]$; this set $\mc{M}$ is an analogue of the positive integers. For elements of $\mathcal{M}$ we have a unique factorization into products of irreducible monic polynomials (prime polynomials). Let $\deg{G}$ denote the degree of $G\in \mathbb{F}_q[t].$ 
For completely multiplicative functions $f : \mc{M} \to \{-1,+1\}$ (that is, functions that satisfy $f(G_1G_2)=f(G_1)f(G_2)$ for all $G_1,G_2\in \mc{M}$), it is known (see e.g.~\cite{GrHaSoFF}) that the partial sums  
\[\sigma_f(n):=\sum_{\substack{G \in \mc{M} \\ \deg{G} \leq n}}f(G)\]
behave rather differently from their number field counterparts. In particular, in the Polymath 5 project~\cite{polymath_example} it was observed that if we define the \emph{long sum discrepancy}
\begin{equation}\label{eq_lsdisc}
\mc{D}_f := \sup_{\substack{D\in \mc{M} \\ N \geq 1}} \Big|\sum_{\substack{G \in \mc{M} \\ \deg{G} \leq N}} f(DG)\Big|,
\end{equation}
then the Erd\H{o}s discrepancy question for $\mathcal{D}_f$ has a \emph{negative} answer, in the sense that there exists even a completely multiplicative $f:\mc{M}\to \{-1,+1\}$ such that $\mathcal{D}_f<\infty.$ In fact, without much additional difficulty we can prove the following.
\begin{prop} \label{prop_bddlf_unbddsf}
There are uncountably  many completely multiplicative functions $f : \mc{M} \to \{-1,+1\}$ for which $\mc{D}_f < \infty$. 
\end{prop}

One of the main goals of the present paper is to characterize the boundedness of partial sums of completely multiplicative functions in function fields, discovering along the way difficulties and features that are not present in the integer setting. We apply this to demonstrate that there are natural formulations of the Erd\H{o}s discrepancy problem in function fields that in contrast have an affirmative answer for completely multiplicative sequences (see Theorem~\ref{thm_lexicographic}). Another consequence of our work is further evidence towards the widely-believed conjecture over $\mb{Z}$  that the functions whose discrepancies are of slowest possible growth are ``modified" characters (Conjecture~\ref{conj_extremal}). See Theorems~\ref{thm_edpff} and~\ref{thm_lsdisc_genchar} for a precise statement (and Definition~\ref{def_modified} for the notion of modified characters).

\subsection{Extremizers for the short sum discrepancy}

The main reason why $\mathcal{D}_f$ is not suitably well-behaved in function fields is because long intervals are too coarse to witness discrepancies in a given sequence. More precisely, an interval $\mc{M}_{\leq N} := \{G\in \mathcal{M}:\,\, \deg{G}\leq N\}$ contains too few other intervals $\mc{M}_{\leq n}$; there are only $N+1$ of them, whereas the interval $\mc{M}_{\leq N}$ has size (i.e. number of elements) of order $\asymp q^N$. In contrast, the interval $[1,N]$ in $\mathbb{N}$ contains $N$ intervals of the form $[1,n]$ with $n \in \mb{Z}$. 
At the same time, Tao's proof of the Erd\H{o}s discrepancy problem makes full use of the fact that there are a lot of different subintervals for an interval in $\mathbb{N}$ by showing in fact the unboundedness of the quantity
\begin{align*}
\frac{1}{\log N}\sum_{m\leq N}\frac{1}{m}\Big|\sum_{|n-m|<H}f(n)\Big|^2,    
\end{align*}
where $H=H(N)$ is slowly growing. This suggests that it is natural to look at the corresponding \emph{short sum discrepancy} over function fields:
\begin{align*}
\mc{S}_f := \limsup_{H \to \infty} \limsup_{N \to \infty}  \sup_{\substack{D,G_0 \in \mc{M} \\ \deg{G_0} = N \\ D\in \mathcal{M}}} \Big|\sum_{\substack{G \in \mc{M} \\ \deg{G-G_0} < H}} f(DG)\Big|,
\end{align*}
which is now taken over the family of short intervals $$I_H(G_0):=\{G\in \mathcal{M}:\,\, \deg{G-G_0}<H\}.$$
These short intervals are much more numerous than the corresponding long intervals and thus provide a much more refined scale to measure the fluctuations of the partial sums; there are $\asymp q^N$ of them inside the set of polynomials of degree at most $N$. 

 Note that over the integers the short sum discrepancy is bounded from above in terms of the long sum discrepancy: since the integers are linearly ordered, we get by the triangle inequality that
$$
\limsup_{H\to \infty}\limsup_{N \to \infty}  \Big|\sum_{|n-N|\leq H} f(dn)\Big| \leq 2 \limsup_{N \to \infty} \Big|\sum_{n \leq N} f(dn)\Big|.
$$

Thus, one presumes that the behavior of the short sum discrepancy $\mathcal{S}_f$ is rather similar to that of Erd\H{o}s discrepancy in the integers. Indeed, we show that $\mathcal{S}_f=\infty$ for ``nearly all'' completely multiplicative functions $f:\mathcal{M}\to \{-1,+1\}$, but, in contrast to the integer case, it turns out that there are also a few exceptional functions. Our next theorem gives a complete classification of the cases where $\mathcal{S}_f$ is bounded for a completely multiplicative function $f:\mathcal{M}\to \{-1,+1\}$ (see Definition~\ref{defn2} below for the definition of a short interval character and its length).

\begin{cor}[Short sum discrepancy is bounded only for modified characters of prime power modulus]\label{short_extr}
Let $f: \mc{M} \to \{-1,+1\}$ be completely multiplicative. Then $\mc{S}_f < \infty$ if and only if there is a prime power $P^k\in \mathcal{M}$, a primitive Dirichlet character $\chi$ modulo $P^k$, a short interval character $\xi$, and an integer $j\in \{0,1\}$ such that $f(P') = \chi(P')\xi(P')(-1)^{j \textnormal{deg}(P')}$ for all primes $P' \neq P$. Moreover, if $q$ is odd, we have $\xi\equiv 1$.
\end{cor}

This result is a corollary of the following more general theorem that applies to completely multiplicative functions taking values on the unit circle.

\begin{thm}\label{thm_edpff}
Let $f: \mc{M} \to S^1$ be a completely multiplicative function. Then $\mc{S}_f < \infty$ if and only if there is a prime power $P^k\in \mathcal{M}$, a primitive Dirichlet character $\chi$ modulo $P^k$, a short interval character $\xi$, and a real number $\theta \in [0,1]$ such that $f(P') = \chi(P')\xi(P')e^{2\pi i \theta \deg{P'}}$ for all primes $P' \neq P$.
\end{thm}

Corollary~\ref{short_extr} and Theorem~\ref{thm_edpff} are closely related to the following conjecture on the growth of the partial sums of multiplicative functions on the integers (see~\cite[Section 1]{TaoEDP} and ~\cite[Section 1]{kmt-ruzsa} for some discussion).

\begin{conj}[Extremality of partial sums of modified characters]\label{conj_extremal}
Let $f:\mathbb{N}\to S^1$ be completely multiplicative. Then
\[\Big|\sum_{n\leq x}f(n)\Big|\ll \log x\]
if and only if there exists a non-principal Dirichlet character $\chi$ modulo a prime power $p^k$ such that $f(p')=\chi(p')$ for all $p' \ne p.$ Conversely, for such $f$ there exists a subsequence $N_k\to\infty,$ such that \begin{align}\label{eqq8}
\Big|\sum_{n\leq N_k}f(n)\Big|\gg \log N_k.    
\end{align}
\end{conj}
Little is known towards this conjecture (which contains the Erd\H{o}s discrepancy problem as a special case), apart from the case where $f$ differs from a Dirichlet character (not necessarily of prime power modulus) at only finitely many primes, which was handled in ~\cite[Corollary 1.6]{kmt-ruzsa}. This case contains~\eqref{eqq8}, which was first handled in~\cite{bcc}, and is the easy part of the conjecture. In fact, the best currently known growth rate for partial sums of length $x$ of a completely multiplicative function $f:\mathbb{N}\to S^1$ is of the form $\Omega( (\log \log x)^{c})$, for some explicit $c > 0$~\cite[Theorem 4.1.1]{mcnamara},~\cite[Section 9.4]{HelRad}.

Both Corollary~\ref{short_extr} and Conjecture~\ref{conj_extremal} manifest the same general phenomenon: the smallest possible discrepancy over $\mb{Z}$ and over $\mb{F}_q[t]$ (for short sums) is attained by ``modified" characters to prime power moduli. Over  $\mathbb{F}_q[t]$ with $q$ \emph{even}, we have an interesting low characteristic phenomenon that the set of characters with bounded discrepancy is somewhat larger than in the case of $q$ odd; this eventually stems from Theorem~\ref{LogEllFF1} below.

Notice that while over $\mb{Z}$ the smallest possible partial sums are believed to be of the order $\asymp \log x$, over $\mb{F}_q[t]$ they are $O(1)$. In order to explain this feature, we recall that for the Borwein--Choi--Coons example~\cite{bcc} given by the``modified" character 
\begin{align*}
f_3(p):=\begin{cases}\chi_3(p),\quad &p\ne 3\\
1,&p=3,
\end{cases}    
\end{align*}
the partial sums satisfy
\begin{align*}
\sum_{n\le x}f_3(n)=\sum_{k\le \log x/\log 3}f_3(3)^k\sum_{m\le \lfloor x/3^k\rfloor}\chi_3(m)\ll \log x,
\end{align*}
since the innermost sum is bounded. On the other hand, to construct a sequence $x_k$ on which the partial sums grow with rate $\gg \log x_k$ one exploits the fact that the intervals $[1,M]$ contain a {\it different number} of residue classes modulo $3$, depending on the value of $M\pmod 3$. This is no longer true over $\mathbb{F}_q[t]$, as all short intervals $I_H(G_0)$ contain the same number of residue classes to any modulus $Q$ as soon as $H \geq \deg Q$. This results in a ``logarithmic" drop\footnote{One can also view this residue class uniformity feature as a manifestation of a ``smoother'' summation in $\mb{F}_q[t]$, in contrast to the sharp cutoffs arising in sums over $\mathbb{Z}$. As pointed out in~\cite{tao} (and attributed to Bill Duke), introducing the smoothing weight $\max\{0,1-\tfrac nx\}$ in the Borwein--Choi--Coons example over the integers also results in uniformly bounded partial sums, and thus a ``logarithmic drop'', as $x \ra \infty$.} as far as quantitative statements are concerned. We shall revisit this further in the following subsections.

\subsection{Discrepancy with the lexicographic ordering}
To rectify the aforementioned difference between the settings of function fields and integers from the previous subsection, a natural approach involves replacing the partial ordering of $\mathbb{F}_q[t]$ employed in constructing the sets $\{G\in \mathbb{F}_q[t]: \deg{G}\leq N\}$ used to define $\mathcal{D}_f$ with a \emph{lexicographic ordering} of $\mathbb{F}_q[t]$. This is (a generalization of) an ordering that has been used in the influential work of Liu and Wooley~\cite{wooley-liu} on Waring's problem over $\mathbb{F}_q[t]$. It arises by associating a base $q$ integer expansion to each polynomial in $\mathbb{F}_q[t]$. In order to define this ordering, we must first impose an ordering on $\mathbb{F}_q$. Let $a_0, a_1, \ldots, a_{q-1}$ be an arbitrary ordering of $\mathbb{F}_q$, and define the size $\langle a\rangle \in \{0,1,\ldots,q-1\}$ of $a \in \mb{F}_q$ as
\begin{align*}
\langle a \rangle=k\quad \textnormal{if}\quad a=a_k.    
\end{align*}
Then we extend $\La \cdot\Ra$ to $\mathbb{F}_{q}[t]$ by defining
\begin{align*}
\La b_Nt^{N}+\cdots +b_1t+b_0\Ra=\La b_N\Ra q^N+\cdots + \La b_1\Ra q+\La b_0\Ra.   
\end{align*}
Implicit in this definition is the requirement\footnote{Otherwise, we would have $\langle G + 0 \cdot t^{\deg{G} + 1} \rangle > \langle G\rangle$, which is absurd.} that $\langle 0 \rangle = 0$, and we assume that our ordering always satisfies this property. 
It is clear that the map $\La \cdot \Ra$ is a bijection from $\mathbb{F}_q[t]$ to $\mathbb{N}\cup \{0\}$. Thus it defines a total order $\prec$ on $\mathbb{F}_{q}[t]$ by setting $A\prec B$ if $\La A\Ra <\La B\Ra$, which is a lexicographic order on $\mb{F}_q[t]$ (and thus a natural one to use).

We remark that Liu and Wooley confined themselves to the lexicographic order (denoted $\langle \cdot \rangle_{\xi}$) that arises from ordering $\mathbb{F}_q$ as $0,\xi^0,\ldots, \xi^{q-1}$, where $\xi$ is a fixed generator of $\mathbb{F}_q^{\times}$. Our results apply to this ordering as well as to any other ordering $\langle \cdot \rangle$ with $\langle 0 \rangle = 0$.

Answering a question of Liu and Wooley\footnote{Personal communication; Oberwolfach 2019.}, we are able to show that when one uses the ordering given by $\La \cdot \Ra$, the Erd\H{o}s discrepancy conjecture holds for all completely multiplicative sequences. Thus, defining the \emph{lexicographic discrepancy} as
\begin{align*}
\mathcal{L}_f:=\sup_{\substack{D \in \mathcal{M} \\ N\geq 1}}\Big|\sum_{\substack{G\in \mathcal{M}\\\La G \Ra \leq N}}f(DG)\Big|,    
\end{align*}
we will prove the following.

\begin{thm}[Lexicographic discrepancy of completely multiplicative sequences is always infinite]\label{thm_lexicographic} For any completely multiplicative sequence $f:\mathcal{M}\to S^1$, we have $\mathcal{L}_f=\infty$.
\end{thm}

Thus, for any completely multiplicative function $f:\mc{M}\to S^1,$ the lexicographically ordered partial sums of $f$ satisfy
\[\sup_{N\geq 1}\Big|\sum_{\substack{G\in \mathcal{M}\\\La G \Ra \leq N}}f(G)\Big|=\infty.\] 
We conclude this subsection by mentioning that the lexicographic ordering appears to be a natural ordering also for several other classical problems over $\mathbb{F}_q[t]$ (in particular, for those where partial summation plays a role). 

\subsection{Long sum discrepancy.}

Having formulated our results for $\mathcal{L}_f$ and $\mathcal{S}_f$, we revisit the long sum discrepancy $\mathcal{D}_f$ to study what can be said about its boundedness. 
We provide a classification of all \emph{modified characters} that have bounded $\mc{D}_f$; these functions are defined as follows.

\begin{def1}[Modified characters]\label{def_modified}  
We call a function $f:\mc{M}\to S^1$ a \emph{modified character} if $f$ is completely multiplicative and for some primitive Dirichlet character $\chi\pmod Q$ of some modulus $Q\in \mathcal{M}$ we have $f(P)=\chi(P)$ for all primes $P\nmid Q$, and otherwise $f(P) \in S^1$ for all $P\mid Q$. We also define modified characters on $\mathbb{N}$ analogously.
\end{def1}

 In the integer setting, the class of modified characters contains the class of functions for which Borwein, Choi and Coons~\cite{bcc} showed unboundedness of discrepancy. Indeed, they considered completely multiplicative functions $f:\mathbb{N}\to \{-1,+1\}$ such that for some prime $p$ we have $f(p')=\chi_{p}(p')$ for all $p'\neq p$, where $\chi_p$ is the Legendre symbol $\pmod p$, analogizing the function $f_3$ constructed above. We note that such functions are significantly easier to work with, since the value of $\sum_{n\leq N}f(n)$ is easy to compute given the base $p$ expansion of $N$.  As soon as one studies modified characters to composite moduli, matters are more complicated and a direct computation of the partial sums (both in the integer case and in the function field case) appears very difficult, requiring control of the digital expansion of $N$ in (at least) two different bases simultaneously.  

We prove the following characterization for the discrepancy of modified characters, where $\omega(Q)$ stands for the number of distinct prime divisors of a polynomial $Q\in \mathbb{F}_q[t]$ and $v_2(n)$ is the $2$-adic valuation of $n$.   

\begin{cor}\label{cor_lsdisc1}
Let $f : \mc{M} \to \{-1,+1\}$ be a modified character associated with a primitive character of modulus $Q$. Then $\mathcal{D}_f<\infty$ if and only if one of the following holds:

\begin{enumerate}[label=(\roman*)]
\item  $\omega(Q)=1$. 

\item $\omega(Q) = 2$, and (up to permutation) the primes $P_1,P_2$ dividing $Q$ satisfy:
\begin{itemize}
\item $f(P_1) =-1,$ $f(P_2)=1$, and 
\item $v_2(\deg{P_1}) \geq v_2(\deg{P_2})$.
\end{itemize}
\item $\omega(Q) = 3$, and (up to permutation) the primes $P_1,P_2,P_3$ dividing $Q$ satisfy: 
\begin{itemize}
\item $f(P_1) = f(P_2) = -1$ and $f(P_3) = 1$,
\item $v_2(\deg{P_1}) \neq v_2(\deg{P_2})$, and 
\item $v_2(\deg{P_j}) \geq v_2(\deg{P_3})$ for $j = 1,2$.
\end{itemize}
\end{enumerate}
\end{cor}

We also give a complete characterization in the case where the modified character is complex-valued; here the statement perhaps surprisingly depends on whether or not a certain polynomial associated to $f$ has multiple roots.

\begin{thm}\label{thm_lsdisc_genchar}
Let $f: \mc{M} \to S^1$ be a modified character associated to a primitive character of modulus $Q$ with $\deg{Q}\geq 1$. Define the polynomial $p(z) := \prod_{P|Q} (z^{\deg{P}}-\overline{f(P)})$.\\ 
a) If all the zeros of $p$ have multiplicity $1$, then $$ \Big|\sum_{\substack{G\in \mathcal{M}\\\deg{G}\leq N}} f(G) \Big| \ll_{Q} 1.$$ 
b) If $b\geq 2$ is the highest multiplicity of a zero of $p$, then there is an increasing sequence $\{N_k\}_{k \geq 1}$ such that
$$\Big|\sum_{\substack{G \in \mathcal{M}\\\deg{G}\leq N_k}} f(G)\Big| \asymp_Q N_k^{b-1}.$$
\end{thm}

It is a natural question to ask for classification of all $\pm 1$-valued multiplicative functions in $\mathcal{M}$ with $\mathcal{D}_f<\infty$. Proposition~\ref{prop_bddlf_unbddsf} and Corollary~\ref{cor_lsdisc1} imply that in the case of general functions this appears all but impossible, whereas for the natural class of modified characters we can give a complete characterization.

Theorem~\ref{thm_lsdisc_genchar} is also related to Conjecture~\ref{conj_extremal}. Namely, one expects that for a completely multiplicative function $f:\mathbb{N}\to S^1$ there is an increasing sequence $\{N_k\}_{k\geq 1}$ such that
\begin{align}\label{eqq9}
\Big|\sum_{n\leq N_k}f(n)\Big|\gg \log N_k.    
\end{align}
As mentioned, from~\cite{bcc} it follows that~\eqref{eqq9} holds whenever $f$ is a modified character (with prime power modulus). Theorem~\ref{thm_lsdisc_genchar} verifies an analogous statement over function fields, namely that when $f: \mathcal{M} \ra S^1$ is a modified character for which $\mathcal{D}_f = \infty$, we have 
\begin{align*}
\Big|\sum_{\substack{G\in \mc{M}\\\deg{G}\leq N_k}}f(n)\Big|\gg \log(q^{N_k})\gg N_k.  
\end{align*}

Theorem~\ref{thm_lsdisc_genchar} also reveals an interesting phenomenon about the \emph{spectrum} of different growth rates of discrepancy (and once again confirms the ``logarithmic" drop). It shows that the discrepancy of a modified character on $\mathbb{F}_q[t]$ always grows like $N^d$ for some number $d\in \mathbb{N}\cup \{0\}$; there are no other possible growth rates. It would be interesting to say something about the spectrum of discrepancies for general multiplicative functions on $\mathbb{F}_q[t]$ (or even on $\mathbb{Z}$); however, this seems extremely difficult since in the non-pretentious case the known lower bound on the growth of the discrepancy is very weak (see Section~\ref{sub: proofideas} for relevant definitions).

\section{Strategy of Proofs}\label{sub: proofideas}

As in Tao's resolution of the Erd\H{o}s discrepancy problem~\cite{TaoEDP}, our proofs naturally split into two main parts: the case of \emph{non-pretentious} multiplicative functions and the case of \emph{pretentious} multiplicative functions 
(see Figure~\ref{fig:1}). At various points we are forced to significantly deviate from the treatment in the number field case. 

By pretentious functions we mean multiplicative functions $f:\mathcal{M}\to S^1$ such that for some character $\widetilde{\chi}:\mathcal{M}\to \mathbb{C}$ of bounded conductor the pretentious distance between $f$ and $\widetilde{\chi}$ is bounded (the Granville--Soundararajan pretentious distance can be generalized to function field setting; see~\eqref{eqq36} below). In the integer setting, the relevant characters would be of the form $\widetilde{\chi}(n)=\chi(n)n^{it}$, so a Dirichlet character $\chi$ times an Archimedean character $n\mapsto n^{it}$, for some $t \in \mb{R}$. A key technical point in this paper is that these characters are \emph{not} sufficient for understanding the behavior of the short sum discrepancy.  We thus need a larger set of characters, introduced in the following definition.  

\begin{def1}[Characters in function fields]\label{defn2} A multiplicative function $\chi:\mathcal{M}\to \mathbb{C}$ which is not identically zero is called a \emph{Dirichlet character} of modulus $Q\in \mathcal{M}$ if $\chi(M+Q)=\chi(M)$ for all $M\in \mathcal{M}$ and $\chi(M)=0$ whenever $(Q,M)\neq 1$. We say that $\chi\pmod Q$ is \emph{primitive} if there is no divisor $Q'\mid Q$, $\deg{Q'}<\deg{Q}$ such that for some Dirichlet character $\chi'\pmod{Q'}$ we have $\chi(M)=\chi'(M)$ whenever $(M,Q)=1$. We say that $\chi\pmod Q$ is \emph{principal} if $\chi(M)=1$ whenever $(M,Q)=1$. 

A function $e_{\theta}\mathcal{M}\to S^1$ of the form $e_{\theta}(M):=e(\theta \deg{M})$ for $\theta\in [0,1]$ is called an \emph{Archimedean character}. 

A multiplicative function $\xi:\mathcal{M}\to \mathbb{C}$ which is not identically zero is called a \emph{short interval character} if there exists $\nu\geq 0$ such that $\xi(A)=\xi(B)$ whenever the $\nu+1$ highest degree coefficients of $A$ and $B$ agree. The smallest such $\nu$ is called the \emph{length} $\textnormal{len}(\xi)$ of $\xi$.
\end{def1}

Any of the characters above are multiplicative. The Archimedean characters play much the same role as the characters $n\mapsto n^{it}$ on $\mathbb{N}$. The notion of short interval characters was introduced by Hayes~\cite{Hayes}, and it has no integer analogue.

\subsection{The non-pretentious case} \label{subsec:NonPret}

In the non-pretentious case, the main ingredient that we need is a function field version of Tao's result on two-point logarithmic correlations of multiplicative functions. This was established by the authors in~\cite{KMT_FFChowla}.

\begin{thm} [Two-point logarithmic Elliott conjecture in function fields,~\cite{KMT_FFChowla}]\label{LogEllFF1}
Let $B \in \mb{F}_q[t] \bk \{0\}$ be fixed. Let $f_1,f_2: \mb{F}_q[t] \to \mb{U}$ be multiplicative. Let $N$ be large, and assume that $f_1$ satisfies the non-pretentiousness condition
\begin{align*}
\min_{\substack{M \in \mc{M} \\ \deg{M} \leq W}} \min_{\psi \hspace{-0,1cm}\pmod{M}} \min_{\substack{\xi\,\, \textnormal{short} \\ \textnormal{len}(\xi) \leq N}}\min_{\theta\in [0,1]} \sum_{\substack{P \in \mc{P} \\ \deg{P} \leq N}} \frac{1-\textnormal{Re}(f_1(P)\bar{\psi}(P)\bar{\xi}(P)e_{\theta}(P))}{q^{\deg{P}}} \xrightarrow{N\to \infty}\infty
\end{align*}
for every fixed $W\geq 1$.
Then
\begin{align*}
\frac{1}{N}\sum_{\substack{G \in \mc{M}\\ \deg{G} \leq N}} q^{-\deg{G}} f_1(G)f_2(G+B) =o(1).
\end{align*}
Moreover, if $f_1$is real-valued and $q$ is odd, then the same conclusion follows provided only that
\begin{align*}
\min_{\substack{M \in \mc{M} \\ \deg{M} \leq W}} \min_{\psi \hspace{-0,1cm}\pmod{M}} \min_{\theta\in \{0,1/2\}} \sum_{\substack{P \in \mc{P} \\ \deg{P} \leq N}} \frac{1-\textnormal{Re}(f_1(P)\bar{\psi}(P)e_{\theta}(P))}{q^{\deg{P}}}\xrightarrow{N\to \infty}\infty.
\end{align*}
\end{thm}

\begin{proof} This is~\cite[Theorem 1.5]{KMT_FFChowla}
\end{proof}

\subsection{The pretentious case}

Assuming for the sake of contradiction that a completely multiplicative $f$ has finite short sum discrepancy $\mathcal{S}_f$, Theorem~\ref{LogEllFF1} can be used as in~\cite{TaoEDP}  to achieve the crucial reduction to the case in which $f$ pretends to be a twisted character $G \mapsto \chi\xi e_{\theta}(G)$, where $\chi$ is a primitive Dirichlet character of bounded conductor, $\xi$ is a short interval character and $\theta \in [0,1]$. At this point, after removing the twist $\xi e_{\theta}$ (which is essentially a triviality) we significantly deviate from Tao's analysis in~\cite{TaoEDP}. 

The first step of this different argument is a technique that allows us to pass from the pretentiousness condition
$$
\sum_{P \in \mc{P}} \frac{1-\text{Re}(f(P)\bar{\chi}(P))}{q^{\deg{P}}} < \infty
$$
to the far more restrictive hypothesis
\begin{equation}\label{eq_findiff}
|\{P \in \mc{P} : f(P) \neq \chi(P)\}| < \infty.
\end{equation}
To accomplish this reduction we use a technical device, which we call the ``rotation trick." This general trick played a crucial role in the recent work~\cite{kmt-ruzsa} on multiplicative functions over the integers.

After this reduction to $f$ satisfying~\eqref{eq_findiff}, it remains to treat the case where $f$ is a \emph{modified character} (see Definition~\ref{def_modified} above). At this point it should be noted that the corresponding argument in~\cite{TaoEDP} (see Section 4, from (4.8) onwards) is insufficient, due essentially to the fact that for \emph{any} $H > \deg{Q}$, orthogonality implies
\[
\sum_{\deg{M} < H} \chi(M) = 0,
\]
for any non-principal Dirichlet character modulo $Q$ (and, as discussed above, the analysis of Borwein--Choi--Coons from~\cite{bcc} runs into serious difficulties in the case where $Q$ has several prime factors). However, a more elaborate argument using Ramanujan sums (see Proposition~\ref{prop_lrg_ssd}~below) does permit one to show that for suitable large choices of $H$, and for $N$ large in terms of $H$, one does have
$$
\max_{\substack{G_0 \in \mc{M}\\ \deg{G_0} = N}} \left|\sum_{\substack{G \in \mc{M} \\ \deg{G-G_0} < H}} f(G)\right| \gg_Q H^{(\omega(Q)-1)/2},
$$
where $\omega(Q)$ denotes the number of distinct prime factors  of the modulus $Q$ of the modified character corresponding to $f$. The logarithmic power growth rate is consistent with the discrete pattern witnessed in Theorem~\ref{thm_lsdisc_genchar}. The above enables us to show that if the number of primes at which $f$ and $\chi$ differ exceeds 1 then $\mc{S}_f = \infty$.  The remaining case in which $\omega(Q) = 1$ can be analyzed directly (because the modulus is a power of a single prime), and this is accomplished at the end of Section~\ref{sec_EDP1}, leading to the proof of Theorem~\ref{thm_edpff}.

\subsection{The lexicographic discrepancy result}
It is crucial to remark that the collection of lexicographic intervals is a \emph{refinement} of the collection of short intervals, in the sense that if $\deg{G_0} = N$ and $H < N$ and $\widetilde{G_0}$ is the unique element of $I_H(G_0)$ divisible by $t^{H}$, then we can express
\[
I_H(G_0) = I_H(\widetilde{G_0})=\{G \in \mc{M} : \langle \widetilde{G_0} \rangle \leq \langle G\rangle < \langle {\widetilde{G_0}}\rangle + q^H\}.
\]
Thus, in view of our short interval result (Theorem~\ref{thm_edpff}), a completely multiplicative function $f:\mathcal{M}\to S^1$ that is uniformly bounded on lexicographic intervals must be a modified character to prime power modulus. Our main obstacle is thus to rule out uniform boundedness of lexicographic partial sums for this class of functions, for which the analysis in short intervals is not sufficient. 

To accomplish this, we fully exploit the ``digital" structure of the lexicographic ordering to obtain a recursive relation for partial sums over $\langle G \rangle\leq N$ at a carefully chosen sequence of scales $N$. The construction is somewhat complicated, and we relegate further explanation to the proof of Proposition~\ref{prop_lexico}.

The proof strategy of Theorems~\ref{thm_edpff} and~\ref{thm_lexicographic} is visualized in the following diagram. \\
\ctikzfig{ProofDiagram}
\begin{figure}[h]
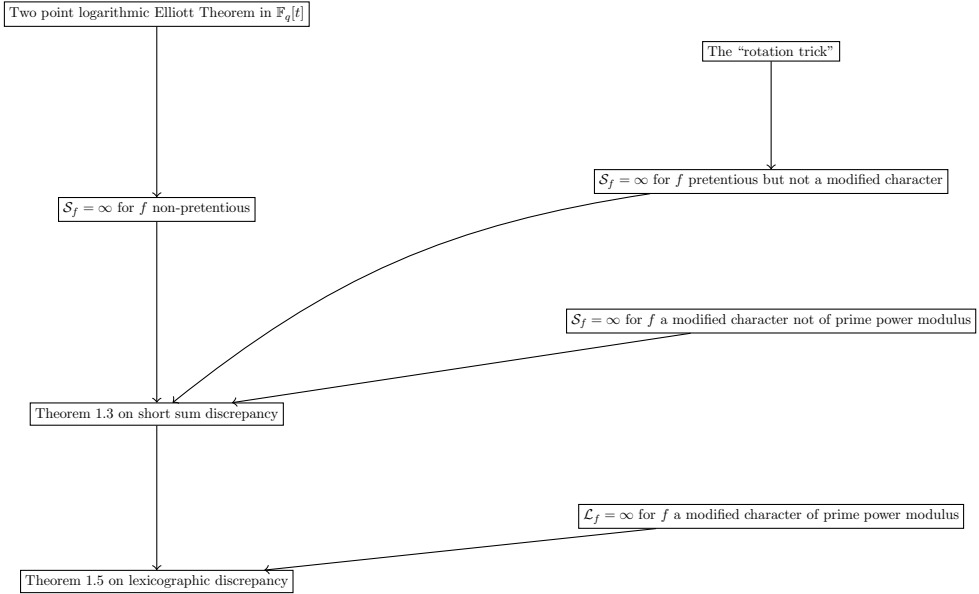
\label{fig:1}
\caption{A diagram describing the different steps of the proofs of Theorems~\ref{thm_edpff} and~\ref{thm_lexicographic}.}
\end{figure}

\subsection{The long sum discrepancy result}
The proof of Theorem~\ref{thm_lsdisc_genchar} proceeds differently than the proofs of Theorems~\ref{thm_edpff} and~\ref{thm_lexicographic}, using a generating function argument and GRH in function fields, and can be read independently.

\subsection{Structure of the Paper}
The paper is organized as follows.  In Section~\ref{sec_EDP1}, we prove the short interval discrepancy theorem (Theorem~\ref{thm_edpff}). In Section~\ref{sec_EDP2}, we establish the discrepancy theorem for lexicographic discrepancy (Theorem~\ref{thm_lexicographic}). Lastly, Section~\ref{sec_EDP3} concerns the long sum discrepancy result, Theorem~\ref{thm_lsdisc_genchar}. 

\subsection{Acknowledgments} This work began when the authors were in residence for the ``Probability in Number Theory'' Workshop at CRM in the spring of 2018, and continued in particular at the ``Sarnak's Conjecture'' workshop at AIM that fall. We would like to thank both institutions for their hospitality and for excellent working conditions. We would also like to thank Yu-Ru Liu and Trevor Wooley for thought-provoking discussions, and Andrew Granville and Maksym Radziwi\l\l \ for their encouragement. We are also grateful to the anonymous referee for helpful comments and suggestions.

The bulk of this paper was written while AM was a CRM-ISM postdoctoral fellow at the Centre de Recherches Math\'{e}matiques in Montr\'{e}al. He would like to warmly thank that institution for its support. OK would like to express his gratitude to Max Planck Institute for Mathematics (Bonn) for providing excellent working conditions and support during the preparation of this manuscript. JT was supported by a Titchmarsh Fellowship, Academy of Finland grant no. 340098, and funding from European Union’s
Horizon Europe research and innovation programme under Marie Sk\l{}odowska-Curie grant
agreement No 101058904.

\section{Notation}\label{sec:notation}

Throughout the paper, $p$ is the characteristic of $\mb{F}_q$, and $q = p^k$ for some $k \geq 1$. 

We denote by $\mc{M}$ the space of monic polynomials in $\mb{F}_q[t]$ (we omit the $q$-dependence in $\mathcal{M}$; thus, whenever $\mathcal{M}$ appears it is understood that the base field has size $q$), and by $\mc{P}$ we denote the space of monic irreducible (prime) polynomials in $\mb{F}_q[t]$. For $N \in \mb{N}$, we write $\mc{M}_N$, $\mc{M}_{\leq N}$ and $\mc{M}_{< N}$ to denote, respectively, the set of monic polynomials of degree exactly $N$, less than or equal $N$ and strictly less than $N$. Analogously, we define $\mc{P}_N$, $\mc{P}_{\leq N}$ and $\mc{P}_{<N}$ to be the corresponding sets of monic irreducible polynomials. We denote the degree of $M \in \mb{F}_q[t]$ by $\text{deg}(M)$.

Given two polynomials $F,G\in \mathbb{F}_q[t]$, not both zero, we define their greatest common divisor $(F,G)$ as the unique monic polynomial $D\in \mathcal{M}$ such that $D\mid F, D\mid G$ and such that for any $D'\in \mathcal{M}$ satisfying $D'\mid F, D'\mid G$ we have $D'\mid D$. The least common multiple $[F,G]$ of $F$ and $G$ is in turn defined by $[F,G]:= FG/(F,G)$.

Typically, $G$ will be used to denote an element of $\mc{M}$, whereas $R$ or $P$ denotes an element of $\mc{P}$ and $M$ denotes an element of $\mb{F}_q[t]$, monic or otherwise.

Given two polynomials $G_0,G \in \mc{M}$ and a parameter $H \geq 1$, we write 
$$
I_H(G_0):= \{G \in \mc{M} : \deg{G-G_0} < H\}
$$ 
to denote the short interval centred at $G_0$ of size $H$.

As usual, given $t \in \mb{R}$ we write $e(t) := e^{2\pi i t}$. Given a parameter $\theta \in [0,1]$ and a polynomial $G \in \mb{F}_q[t]$, we also write $e_{\theta}(G) := e(\theta \deg{G})$. 
Finally, given a rational function $F/G \in \mb{F}_q(t)$ with Laurent series expansion $F/G = \sum_{k = -N}^{N'} a_k t^{k}$, we define $e_{\mb{F}}(\alpha) := e((\text{tr}_{\mb{F}_q/\mb{F}_p} a_{-1}(\alpha))/p)$, where $\text{tr}_{\mb{F}_q/\mb{F}_p}$ denotes the usual field trace.

Throughout the paper, we write $\mb{U} := \{z \in \mb{C} : |z| \leq 1\}$ and $S^1 := \{z \in \mb{U} : |z| = 1\}$. Given sequences $f,g: \mc{M} \to \mb{U}$,  we define the pretentious distance between them by
\begin{align}\label{eqq36}
\mb{D}(f,g;N) := \left(\sum_{P \in \mc{P}_{\leq N}} q^{-\deg{P}}(1-\text{Re}(f(P)\bar{g}(P)))\right)^{1/2},
\end{align}
and also set
$$
\mc{D}_f(N) := \min_{\theta \in [0,1]} \mb{D}(f, e_{\theta};N)^2.
$$ 
We frequently use the pretentious triangle inequalities (see e.g.~\cite[Section 2]{KluThe}):  for any functions $f_1,f_2,f_3,f_4:\mathcal{M}\to \mathbb{U}$, we have 
\begin{align}\label{eq:pret1} \mathbb{D}(f_1,f_3;N)\leq \mathbb{D}(f_1,f_2;N)+\mathbb{D}(f_2,f_3;N)
\end{align}
and 
\begin{align}\label{eq:pret2}
\mathbb{D}(f_1f_2,f_3f_4;N)\leq \mathbb{D}(f_1,f_3;N)+\mathbb{D}(f_2,f_4;N).        
\end{align}

For $f: \mc{M} \to \mb{U}$ a 1-bounded multiplicative function, we define the Dirichlet series corresponding to $f$ by
\begin{align}\label{eqq19}
L(s,f) := \sum_{G \in \mc{M}} f(G)q^{-\deg{G}s} = \prod_{P \in \mc{P}} \sum_{k \geq 0} f(P^k)q^{-k\deg{P}s},
\end{align}
for $\textnormal{Re}(s)>1$; in this region both expressions converge absolutely.

We will sometimes write $\mu_k$ to denote the set of $k$th order roots of unity, where $k \in \mb{N}$. 

The functions $\Lambda$,  $\omega$, $\lambda$, $\mu$, $\phi$, $\textnormal{rad}$ and $\nu_P$, defined on $\mc{M}$, are the analogues of the corresponding arithmetic functions in the number field setting. Thus 
\begin{itemize}
    \item $\Lambda(G)=\deg{P}$ if $G=P^k$ for some $k\geq 1$ and  $P\in \mathcal{P}$ and $\Lambda(G)=0$ otherwise. 
    \item $\omega(G)$ is the number of distinct irreducible divisors of $G$.
    \item $\lambda:\mathcal{M}\to \{-1,+1\}$ is the completely multiplicative function with $\lambda(P)=-1$ for all $P\in \mathcal{P}$.
    \item $\mu:\mathcal{M}\to \{-1,0,+1\}$ is given by $\mu(G)=(-1)^{\omega(G)}$ for $G$  not divisible by $P^2$ for any $P\in \mathcal{P}$, and $\mu(G)=0$ otherwise. 
    \item $\phi(G)$ is the size of the finite multiplicative group $(\mathbb{F}_q[t]/G\mathbb{F}_q[t])^{\times}$.
    \item $\textnormal{rad}(G)=1$ if $G=1$ and $\textnormal{rad}(G)=P_1\cdots P_k$ if $P_1,\ldots, P_k$ are the distinct irreducible factors of $G$.
    \item $\nu_P(G)$, for $P\in \mathcal{P}$, is the largest integer $k$ such that $P^k\mid G$.
\end{itemize}
Recall also Definition~\ref{defn2} for the definitions of the various types of characters used in this paper. 

Throughout this paper, the cardinality $q$ of the underlying finite field $\mb{F}_q$ is fixed. For the sake of convenience we have chosen to omit mention of dependencies on $q$ of implicit constants in our estimates. In particular, the implicit constants in any estimate may depend on $q$ throughout this paper.

\section{The Short Sum Discrepancy} \label{sec_EDP1}

The proof of Theorem~\ref{thm_edpff} will be achieved through a series of reductions, starting with a reduction to the case of functions that pretend to be characters.

\subsection{Reduction to the pretentious case}

In this subsection, we will show that the short sum discrepancy $\mathcal{S}_f$ of $f$ is infinite whenever $f$ is non-pretentious in a suitable sense. Recall the notions of characters and pretentiousness in this context from Definition~\ref{defn2} and the notation section, respectively.

\begin{prop}\label{prop_redtoDirPret}
Let $f: \mc{M} \to S^1$ be a completely multiplicative function, and let $C > 0$. Assume that
\begin{equation}\label{CBound}
\mathcal{S}_f = \limsup_{H\to \infty } \limsup_{N \to \infty}  \max_{G_0 \in \mc{M}_{N}} \left|\sum_{G \in I_H(G_0)} g(G)\right| \leq C.
\end{equation}
Then there is a primitive Dirichlet character $\chi$ with $\textnormal{cond}(\chi) \ll_C 1$, a short interval character $\xi$ of length $\ll_C 1$ and a real number $\theta \in [0,1]$ such that $\mb{D}(g,\chi\xi e_{\theta};N) \ll_C 1$ for all $N\geq 1$. 

Moreover, in the case that $f: \mc{M} \ra \{-1,+1\}$ we may conclude, in fact, that there is a real primitive character $\chi$, a real short interval character $\xi$ and $\theta \in \{0,1/2\}$ such that $\mb{D}(f,\chi \xi e_{\theta};N) \ll_C 1$ for all $N\geq 1$. If $q$ is additionally odd, we can also say that $\xi\equiv 1$.
\end{prop}

\begin{proof}
Let $H=H(C)$ be an integer large enough in terms of $C$. We may assume that $N$ is large enough in terms of $H$, so that in particular $1 \leq H < \sqrt{\log N}/10$. We may bound $\mathcal{S}_f^2$ from below by a (logarithmically-weighted) $L^2$-average of sums over intervals $I_H(G_0)$ with $G_0 \in \mc{M}_{\leq N}$ to deduce that
$$
\frac{1}{N+1} \sum_{G_0 \in \mc{M}_{\leq N}\setminus \mathcal{M}_{\leq \sqrt{\log N}}} q^{-\deg{G_0}}\left|\sum_{G \in I_H(G_0)} f(G)\right|^2 \leq (\mathcal{S}_f+1)^2\ll_C 1.
$$

We expand the square, exchange orders of summation and separate the terms according to $\deg{G_0}=m$. We obtain
\begin{align}\label{eq:sqrtlog}
\frac{1}{N+1}\sum_{\sqrt{\log N} < m \leq N} q^{-m} \sum_{G_0 \in \mc{M}_m} \sum_{\deg{M_1},\deg{M_2} < H} f(G_0+M_1)\bar{f(G_0+M_2)} \ll_C 1.
\end{align}
For each $m \geq H$ and $M_1\in \mathcal{M}$ of degree $<H$, we have $\mc{M}_m + M_1 = \mc{M}_m$. Thus, making on the left-hand side of~\eqref{eq:sqrtlog} the change of variables $G := G_0 + M_1$ and $M := M_2-M_1$ (so that $\deg{M} < H$), and then bounding the contribution from the summands with $m < \sqrt{\log N}$ trivially as $O(q^{2H}\sqrt{\log N})$, we reach
$$
\frac{|\mc{M}_{<H}|}{N+1} \sum_{m \leq N} q^{-m} \sum_{G \in \mc{M}_m} \sum_{\deg{M} < H} f(G)\bar{f(G+M)} \ll_C 1 + q^{2H}\sqrt{\log N}/N.
$$
If we isolate the choice $M = 0$ from the remaining shifts, we deduce that
$$
\left|\sum_{\substack{\deg{M} < H \\ M \neq 0}}\frac{1}{N+1} \sum_{m \leq N} q^{-m}  \sum_{G \in \mc{M}_m} f(G)\bar{f(G+M)}\right| \gg_C   1 - O_C(\sqrt{\log N}/N).
$$
By the assumption that $N$ is large in terms of $H$ , the triangle inequality and the pigeonhole principle then imply that for some $M \neq 0$ with $\deg{M} < H$ we actually have
$$
\left|\frac{1}{N}\sum_{m \leq N} q^{-m} \sum_{G \in \mc{M}_m} f(G)\bar{f(G+M)}\right|  \gg_C q^{-H}\gg_C 1.
$$

By the first statement in Theorem~\ref{LogEllFF1}, we conclude that there exists a primitive Dirichlet character $\chi_N$ with $\text{cond}(\chi_N) = O_C(1)$, a primitive short interval character $\xi_N$ of length $\leq N$ and a point $\theta_N \in [0,1]$ such that $\mb{D}(f,\chi_N \xi_N e_{\theta_N};N) \ll_C 1$. Note that the set of Dirichlet characters of conductor at most $O_C(1)$ is bounded in size (in terms of $C$) and that $[0,1]$ is compact. There is thus an infinite increasing sequence $\{N_j\}_j$ of positive integers, a primitive character $\chi$ of conductor $\ll_C 1$ and a $\theta \in [0,1]$ for which $\theta_{N_j} \to \theta$ as $j \to \infty$, such that 
$$
\limsup_{j \to \infty} \mb{D}(f,\chi \xi_j e_{\theta_j}; N_j) < \infty,  
$$
where by an abuse of notation we have written $\xi_j := \xi_{N_j}$ and $\theta_j := \theta_{N_j}$, for convenience. 

Since $N_{j+1} > N_j$, it follows from the triangle inequality~\eqref{eq:pret1} that
\begin{equation} \label{eq_compare}
\mb{D}(\xi_j e_{\theta_j}, \xi_{j+1} e_{\theta_{j+1}};N_{j}) \leq \mb{D}(f,\chi \xi_j e_{\theta_j}; N_j) + \mb{D}(f,\chi \xi_{j+1} e_{\theta_{j+1}}; N_{j+1}) \ll_C 1
\end{equation}
uniformly in $j\geq 1$. Now, suppose $\xi_j \bar{\xi}_{j+1}$ is nontrivial. Note that $\xi_j\bar{\xi_{j+1}}$ is a short interval character of length $\leq \max\{\text{len}(\xi_j),\text{len}(\xi_{j+1})\} \leq N_j$. But then by~\eqref{eqq36} and~\cite[Lemma 3.2]{KMT_FFChowla},
\begin{align}\label{eq:logN}
\mb{D}(\xi_je_{\theta_j},\xi_{j+1}e_{\theta_{j+1}};N_j)^2 &= \log N_j - \text{Re}\left(\sum_{P \in \mc{P}_{\leq N_j}} \xi_j\bar{\xi}_{j+1}(P)e_{\theta_j-\theta_{j+1}}(P) q^{-\deg{P}}\right) + O(1) \\
&\geq (1-o(1)) \log N_j.
\end{align}

This contradicts~\eqref{eq_compare}. It must follow that $\xi_j = \xi_{j+1}$ for all $j$ sufficiently large (in terms of $C$). In particular, it follows that there is a $j_0 \ll_C 1$ such that $\xi_j = \xi_{j_0}$ for all $j \geq j_0$. Setting $\xi := \xi_{j_0}$, which is a short interval character of length $\ll_C 1$, we deduce that $\mb{D}(f,\chi \xi e_{\theta_j};N_j) \ll_C 1$ for all $j\geq 1$, and therefore by the triangle inequality \eqref{eq:pret1} it follows that
$$
\mb{D}(e_{\theta_j},e_{\theta_{j+k}};N_j) \leq \mb{D}(f,\chi \xi e_{\theta_j};N_j) + \mb{D}(f,\chi \xi e_{\theta_{j+k}};N_{j+k})  \ll_C 1,
$$
uniformly in $j,k \geq 1$. Since the expression on the left-hand side is continuous in $\theta_{j+k}$, taking $k \to \infty$ we deduce that
$$
\mb{D}(e_{\theta_j},e_{\theta};N_j) \ll_C 1,
$$
uniformly in $j$, and hence for all $j\geq 1$ we have
\begin{align*}
\mb{D}(f,\chi \xi e_{\theta};N_j) &\leq \mb{D}(f,\chi \xi e_{\theta_j}; N_j) + \mb{D}(\chi \xi e_{\theta},\chi \xi e_{\theta_j};N_j)\\
&=\mb{D}(f,\chi \xi e_{\theta_j}; N_j) + \mb{D}(e_{\theta}, e_{\theta_j};N_j)+O_C(1)\ll_C 1,
\end{align*}
where we used the fact that $|\chi \xi(P)|=1$ for all $P\in \mathcal{P}$ of degree $\gg_C 1$. 
 
 Since $\mb{D}(f,\chi \xi e_{\theta}; N_j) \leq \mb{D}(f,\chi \xi e_{\theta};N) \leq \mb{D}(f,\chi \xi e_{\theta}; N_{j+1})$ whenever $N_j \leq N < N_{j+1}$, we deduce that
\begin{align}\label{eq_conclusion}
\mb{D}(f,\chi \xi e_{\theta};N) \ll_C 1,
\end{align}
uniformly in $N$. This completes the proof of the first claim. 

To prove the second claim where $f:\mathcal{M}\to \{-1,+1\}$, we take the conclusion~\eqref{eq_conclusion} and apply the triangle inequality~\eqref{eq:pret2} to deduce that
\begin{align}\label{eq_conclusion2}
\mathbb{D}(1,\chi^2\xi^2e_{2\theta};N)\leq  2\mb{D}(f,\chi \xi e_{\theta};N) \ll_C 1. 
\end{align}
Arguing similarly as in~\eqref{eq:logN}, this implies that  $\chi^2$ is principal and $\xi^2\equiv 1$. Furthermore, in this case, by~\eqref{eqq36} we have
\begin{align}\begin{split}\label{eq:logcos}
\mb{D}(1,\chi^2e_{2\theta};N)^2 &= \mb{D}(1,e_{2\theta};N)^2+O_C(1)\\
&=\log N - \sum_{n \leq N} \frac{\cos(4\pi \theta n)}{n} + O_C(1).
\end{split}
\end{align}
If $2\theta\not \equiv 0\pmod 1$, there exist $\eta>0$ and $B>0$ such that every interval of length $B$ contains an integer $n$ for which 
 $|\cos(4\pi \theta n)|\leq 1-\eta$. Inserting this into~\eqref{eq:logcos} and comparing with~\eqref{eq_conclusion2}, we conclude that we must have $2\theta \equiv 0 \pmod{1}$.

Lastly, if $q$ is odd, then by the statement in Theorem~\ref{LogEllFF1} about real-valued $f$, we also have $\xi\equiv 1$ (in fact, there are no nontrivial real-valued short interval characters then). The second claim thus follows.
\end{proof}

\subsection{Reduction to modified characters}

We have demonstrated that in order to characterize those completely multiplicative functions $f$ with bounded short sum discrepancy $\mathcal{S}_f$, it suffices to treat functions that are pretentious to a twisted character. By means of the following proposition, however, we can in fact restrict ourselves to functions differing from a twisted character at a \emph{bounded} number of irreducibles, only. The proof of the proposition utilizes what we call a ``rotation trick''; see~\cite{kmt-ruzsa} for applications of the same idea in the integer setting.

\begin{prop}\label{prop_redtofin}
Let $f: \mc{M} \to S^1$ be completely multiplicative. Suppose there exist $Q \in \mc{M}$ and a primitive Dirichlet character $\chi$ modulo $Q$, a primitive short interval character $\xi$ of length $\nu\geq 0$ and $\theta \in [0,1]$ such that $\mb{D}(f,\chi\xi e_{\theta};\infty) < \infty$. Let 
$$
S := \{P \in \mc{P} : f(P) \neq \chi(P)\xi(P)e_{\theta}(P)\}.
$$ 
If $|S| = \infty$ then $\mathcal{S}_f=\infty$.
\end{prop}

\begin{proof}

Let $f:\mathcal{M}\to S^1$ be  be completely multiplicative. Assume for the sake of contradiction that $|S| = \infty$. 
We will prove that
\begin{align}\label{eqq13}
\limsup_{H\to \infty}\limsup_{N \to \infty}  \max_{G_0 \in \mc{M}_N} \left|\sum_{G \in I_H(G_0)} f(G)\right| = \infty.
\end{align}
 For $N$ large enough in terms of $H$, both $e_{\theta}$ and $\xi$ are constant on any short interval $I_H(G_0)$ with $G_0\in \mathcal{M}_N$, so we may replace $f$ by $f\bar{\xi}e_{-\theta}$ in~\eqref{eqq13} and (still calling this new function $f$ for convenience) we may assume that $\mathbb{D}(f,\chi;\infty)<\infty$ and that $S:=\{P\in \mathcal{P}:\,\, f(P)\neq \chi(P)\}$ is infinite.

Let $1 \ll H \ll n \ll N$ be parameters, each of which is large enough in terms of the parameters to the left of it. Since $\mathbb{D}(f,\chi;\infty)<\infty$, we can impose the condition
\begin{align}\label{eq:Dcondition}
\mb{D}(f,\chi; n,\infty)^2 &:= \sum_{\substack{P \in \mc{P} \\ \deg{P} > n}} \frac{1-\text{Re}(f(P)\bar{\chi}(P))}{q^{\deg{P}}} \leq 1/q^{3H}.
\end{align}
Since $|S|=\infty$, there exists a function $F:\mathbb{N}\to \mathbb{R}_{\geq 0}$, depending only on $H$, such that there are $> q^{2H}$ primes $P \in S$ with $\deg{P} \in (n,F(n)]$. 

For each $M \in \mb{F}_q[t]$ of degree $<H$ pick some $P_M \in S$ such that the $P_M$ are all distinct and such that $\deg{P_M} \in (n,F(n)]$, and let $k_M\ll_{H,n}1$ be a positive integer to be chosen later. We set 
$$
\Gamma := \prod_{G \in \mc{M}_{\leq n}} G^2 \cdot \prod_{\deg{M} < H} P_M^{k_M};
$$
note that 
\begin{align*}
\deg{\Gamma} &\leq 2\sum_{G \in \mc{M}_{\leq n}} \deg{G} + \sum_{\deg{M} < H} k_M \deg{P_M}\\
&\leq 8nq^n+ \left(\max_{\deg{M} < H} k_M\right)q^HF(n) \\
&\leq (\log N)/(50 \log q),
\end{align*}
if $H$ is large enough and $N$ is large enough in terms of $n$ and $H$. 

By the Chinese remainder theorem, we can choose $R\in \mathcal{M}_{2\deg(\Gamma)}$ such that
\begin{align*}
R &\equiv 0 \pmod{\prod_{G \in \mc{M}_{\leq n}} G^2} \\
R &\equiv - M+P_M^{k_M} \pmod{P_M^{k_M+1}}
\end{align*}
for all $M\in \mathcal{M}$ with $\deg{M} < H$.

Note that if $M\neq M'$ are in $\mathbb{F}_q[t]$ and $\deg{M},\deg{M'}<H$, then $P_{M'} \nmid (R+M)$, since otherwise $P_{M'}\mid (M'-M)$ but $\deg{P_{M'}}>H$. Therefore,
\begin{align}\label{eqq11}
(R+M,\Gamma) = (M,\Gamma) \cdot P_M^{k_M} = \widetilde{M}P_M^{k_M},    
\end{align} 
where, setting $a$ to be the leading coefficient of $M$, we put $\widetilde{M}(t):=M(t)/a$ .

We consider the double sum
$$
\Sigma := q^{-N}\sum_{\deg{M} < H} \sum_{G \in \mc{M}_N} f(G \Gamma + R + M).
$$
Since $\mc{S}_f < \infty$, swapping the orders of summation, summing in $M$ and applying the triangle inequality, we see that
$$
|\Sigma| \leq q^{-N} \sum_{G \in \mc{M}_N} \left|\sum_{\deg{M} < H} f(G\Gamma + R + M)\right| \leq \mc{S}_f+1 < \infty.
$$

Now fix $M\in \mathcal{M}_{<H}$ for the moment. Set 
$$d_M := a(R+M,\Gamma)=MP_M^{k_M},\quad  \Gamma_M := \Gamma/d_M.$$ Factoring out primes in common with $\Gamma$ and noting that $\deg{R+M}<N$, we have
\begin{align*}
\sum_{G \in \mc{M}_N} f(G\Gamma + R+M) &= f(d_M) \sum_{G \in \mc{M}_N} f\left(G\Gamma_M + \frac{R+M}{d_M}\right) \\
&= f(d_M) \sum_{G' \in \mc{M}_{N+\deg{\Gamma_M}}} f(G') 1_{G' \equiv (R+M)/d_M \pmod{\Gamma_M}}.
\end{align*}
Using orthogonality of Dirichlet characters modulo $\Gamma_M$, the above expression equals to
$$
\frac{f(d_M)}{\phi(\Gamma_M)} \sum_{\psi \pmod{\Gamma_M}} \psi((R+M)/d_M) \sum_{G' \in \mc{M}_{N+\deg{\Gamma_M}}} f(G')\bar{\psi}(G').
$$
Choosing $n$ large enough in terms of $H$, we can guarantee that $Q|\Gamma_M$, regardless of $M$. Thus, there is a character $\chi'$ modulo $\Gamma_M$ that is induced by $\chi$. If $\psi \neq \chi'$ then, provided $N$ is large enough in terms of $n$,~\cite[Corollary 3.7]{KMT_FFChowla} yields
\begin{align*}
\max_{\deg{M} < H}\, \max_{\substack{\psi \pmod{\Gamma_M} \\ \psi \neq \chi'}} \left|\sum_{G' \in \mc{M}_{N+\deg{\Gamma_M}}} f(G') \bar{\psi}(G')\right| \ll q^{N+\deg{\Gamma}} N^{-1/4 + o(1)} \ll q^N/N^{1/5},
\end{align*}
since $\deg{\Gamma} \leq \frac{\log N}{50\log q}$. Thus,
\begin{align} \label{eq_main_term_del}\begin{split}
&q^{-N} \sum_{G \in \mc{M}_N} f(G\Gamma+R+M)\\
&= f(d_M)\chi'\left(\frac{R+M}{d_M}\right) \frac{q^{\deg{\Gamma_M}}}{\phi(\Gamma_M)} q^{-N-\deg{\Gamma_M}} \sum_{G' \in \mc{M}_{N+\deg{\Gamma_M}}} f(G')\bar{\chi'}(G') + O(N^{-1/5}).
\end{split}
\end{align}
Observe that using~\eqref{eqq11},
$R/M\equiv 0\pmod{Q}$ and $(R/M+1,\Gamma_M) = (P_M^{k_M},\Gamma_M) = 1$, we have
\begin{align}\label{eqq12}
\chi'\left(\frac{R+M}{d_M}\right)=\chi'\left(\frac{R/M+1}{P_M^{k_M}}\right)=\bar{\chi}(P_M)^{k_M}.    
\end{align}

Applying Delange's theorem over function fields to $f\bar{\chi'}$ (see~\cite[Theorem 1.4.1]{KluThe}), and recalling that $\chi'(P)=0$ if $P\mid \Gamma_M$, we see that
\begin{align}
&\sum_{G' \in \mc{M}_{N+\deg{\Gamma_M}}} f(G') \bar{\chi'}(G') \nonumber\\
&= q^{N+\deg{\Gamma_M}} \Bigg(\frac{\phi(\Gamma_M)}{q^{\deg{\Gamma_M}}} \prod_{\substack{P \in \mc{P} \\ \deg{P} > n \\ P \neq P_{M'} \, \forall M':\, \deg{M'} < H, M' \neq M}} \left(1-q^{-\deg{P}}\right)\left(1-f\bar{\chi'}(P)q^{-\deg{P}}\right)^{-1}\nonumber\\
& \quad \quad \quad \quad \quad \quad \quad + O\Bigg(\mb{D}(f,\chi; \frac{\log N}{2\log q},\infty) + N^{-1/2}\Bigg)\Bigg) \nonumber\\
&= \frac{\phi(\Gamma_M)}{q^{\deg{\Gamma_M}}} q^{N+\deg{\Gamma_M}}\Bigg(1+O\Bigg(\sum_{\substack{P\in \mathcal{P}\\\deg{P}>n}}\frac{1-\Re(f(P)\bar{\chi'}(P))}{q^{\deg{P}}} + q^{-3H/2} + N^{-1/2}\Bigg)\Bigg) \nonumber\\
&= \frac{\phi(\Gamma_M)}{q^{\deg{\Gamma_M}}} q^{N+\deg{\Gamma_M}}\left(1+O(q^{-3H/2}+ N^{-1/2})\right) \label{eq_bd_mainterm}
\end{align}
by~\eqref{eq:Dcondition}, provided that that $n$ is large enough in terms of $H$. 
Since the above can be done uniformly over all $\deg{M} < H$, we deduce upon inserting~\eqref{eq_bd_mainterm},~\eqref{eqq12} and~\eqref{eqq11} into~\eqref{eq_main_term_del} that when $N$ is large enough relative to $H$,
\begin{align*}
\mathcal{S}_f+1&\geq \Sigma =  q^{-N}\sum_{\deg{M} < H} \sum_{G \in \mc{M}_N} f(G\Gamma + R + M) = \sum_{\deg{M} < H} f(d_M) \bar{\chi}(P_M)^{k_M} + o(1)\\
&= \sum_{\substack{\deg{M} < H}} f\bar{\chi}(P_M)^{k_M} \prod_{\deg{P} \leq n} f(P)^{\nu_P(M)}+o(1)\\
&= \sum_{\deg{M} < H} f\bar{\chi}(P_M)^{k_M} f(M) + o(1).
\end{align*}

We now show that there is a choice of the multiplicities $k_M$ that makes
\begin{align*}
\left|\sum_{\deg{M} < H} f\bar{\chi}(P_M)^{k_M} f(M)\right|\geq q^H/10, 
\end{align*}
say, which will provide the desired contradiction for $H$ large enough. This follows from the following lemma.

\begin{lem} Let $m\geq 1$, let $w_1,\ldots, w_m\in S^1$, and let $\zeta_1,\ldots, \zeta_m\in S^1\setminus\{1\}$. Then there exist $k_j\in \mathbb{N}$ such that
\begin{align}\label{eqq14}
\left|\sum_{j\leq m}\zeta_j^{k_j}w_j\right|\geq  m/7.
\end{align}
\end{lem}

\begin{proof}
By the pigeonhole principle, there exists a closed arc $I$ of the unit circle $S^1$ of length $2\pi/3$ that contains $\geq m/3$ of the complex numbers $w_j$. Let $J$ be the set of $j$ for which $w_j\in I$. Form a semicircle $\mathcal{C}\subset S^1$ such that $I\subset \mathcal{C}$ and such that the midpoint of $I$ is the midpoint of the arc of $\mathcal{C}$. 

Now, for every $j\in J$, pick $k_j$ such that $|\zeta_j^{k_j}-1|\leq 1/(100m)$. For every $j\in \{1,\ldots, m\}\setminus J$, pick $k_j$ such that $\zeta_j^{k_j}w_j\in \mathcal{C}$; this is clearly always possible since $\{\zeta_j^{k}:\,\, k\in \mathbb{N}\}$ intersects any semicircle. Let $\alpha\in S^1$ be such that the half-plane determined by $\mathcal{C}$ is $\{z\in \mathbb{C}:\,\, \Re(\alpha z)\geq 0\}$. Note that $\Re(\alpha z)\geq \frac{1}{2}$ whenever $z\in I$. Thus
\begin{align*}
\left|\sum_{j\leq m}\zeta_j^{k_j}w_j\right|&\geq \Re\left(\alpha \sum_{j\leq m}\zeta_j^{k_j}w_j\right)\\
&\geq \Re\left(\alpha \sum_{j\in J}\zeta_j^{k_j}w_j\right)\\
&\geq \frac{m}{3}\cdot \left(\frac{1}{2}-\frac{1}{100}\right)\\
&\geq  \frac{m}{7},
\end{align*}
which proves the claim 
\end{proof}

Taking $\zeta_M=f\bar{\chi}(P_M)\neq 1$ and $w_M=f(M)$ in the lemma, a choice of multiplicities $k_M\ll_{n,H}1$ can be made, and the claim follows.
\end{proof}

\subsection{The case of modified characters}

It now remains to consider functions that differ at only finitely many primes from a non-principal Dirichlet character. Indeed, as was noted in the proof of Proposition~\ref{prop_redtofin}, if 
\begin{align*}
\limsup_{H\to \infty}\limsup_{N \to \infty} \max_{G_0 \in \mc{M}_{N}} \left|\sum_{G \in I_H(G_0)} f(G)\right| = \infty,
\end{align*}
holds for a function $f$, then it also holds for the function $f\bar{\xi}e_{-\theta}$, so we may assume by Proposition~\ref{prop_redtofin} that $|\{P\in \mathcal{P}:\,\, f(P)\neq \chi(P)\}|<\infty$. 

This is precisely the case of modified characters (see Definition~\ref{def_modified} above).

\begin{rem}\label{rem_genchar}
Note that if $f$ differs from a non-principal Dirichlet character $\chi'$ at only finitely many primes $S$, say, then by setting $\chi := \chi' \chi_0^{(Q_S)}$, where $Q_S := \prod_{P\in S} P$ and $\chi_0^{(Q_S)}$ denotes the principal character modulo $Q_S$, then $f$ is a modified character modulo $[Q,Q_S]$.
\end{rem}

\subsubsection{Modified characters with at least two prime factors}

The last major ingredient that we require before proceeding to the proof of Theorem~\ref{thm_edpff} involves showing that modified characters have unbounded short sum discrepancy, provided the modulus has at least two distinct prime factors. We start with a lemma that will be used subsequently.

\begin{lem}
For a Dirichlet character $\chi\pmod Q$ with $Q\in \mathcal{M}$ we define the \emph{Gauss sum}\footnote{Recall the definition of the exponential function $e_{\mathbb{F}(\cdot)}$ in $\mathbb{F}_q(t)$ from Section~\ref{sec:notation}.}
\begin{align*}
\tau(\chi)=\sum_{A\pmod Q}\chi(A)e_{\mathbb{F}}\left(\frac{A}{Q}\right).    
\end{align*}
Then $\tau(\chi)\neq 0$ whenever $\chi\pmod Q$ is primitive and non-principal.
\end{lem}

\begin{proof}
If $\chi\pmod{Q}$ is primitive and non-principal, the same argument as in the integer case (see~\cite[Section 2]{davenport}) shows that $|\tau(\chi)|=q^{\deg{Q}/2}$, so the claim follows.
\end{proof}
We will also need the following formula for the Gauss sums
$$
\tau(\chi,B) := \sum_{A \pmod{Q}} \chi(A) e_{\mb{F}}\left(\frac{AB}{Q}\right),
$$
particularly when $\chi$ is imprimitive.
\begin{lem} \label{lem_imprim}
Let $Q = Q_1Q_2 \in \mc{M}$, where $(Q_1,Q_2) = 1$ and $Q_2$ is squarefree. Let $\chi$ be a character modulo $Q$, induced by a primitive character $\chi^{\ast}$ modulo $Q_1$. Then for any non-zero $B \in \mb{F}_q[t]$,
$$
\tau(\chi,B) = \tau(\chi^{\ast}) \chi^{\ast}(Q_2)\overline{\chi^{\ast}}(B)\phi((Q_2,B))\mu(Q_2/(Q_2,B))1_{(Q,B)|Q_2}.
$$
\end{lem}
\begin{proof}
Following the proof of~\cite[Lemma 5.4]{MVgold} in the function field setting, we find that
$$
\tau(\chi,B) = \tau(\chi^{\ast})\bar{\chi^{\ast}}(B/(Q,B)) \frac{\phi(Q)}{\phi(Q/(Q,B))} \mu(Q_2/(Q,B)) \chi^{\ast}(Q_2/(Q,B))
$$
if $(Q,B)|Q_2$, and $\tau(\chi,B)=0$ otherwise. We focus on the former case. Since $(Q_1,Q_2) = 1$ and $Q_2$ is squarefree, $B$ is coprime to $Q_1 = \text{cond}(\chi)$, $(Q,B) = (Q_2,B)$ and we can simplify the character factors to give $\chi^{\ast}(Q_2)\overline{\chi^{\ast}}(B)$. Furthermore, we have 
$$
\frac{\phi(Q)}{\phi(Q/(Q,B))} = \frac{\phi(Q_2)}{\phi(Q_2/(Q_2,B))} = \phi((Q_2,B)),
$$ 
which implies the claim.
\end{proof}

Now the result about modified characters modulo $Q\neq P^r$ follows in a strong form from the following result.

\begin{prop} \label{prop_lrg_ssd}
Let $f: \mc{M} \to  S^1$ be a modified character modulo $Q\in \mathcal{M}$, associated with a non-principal character $\chi$, induced by a primitive character $\chi^{\ast}$ modulo $Q^{\ast}$. Assume moreover that $Q/Q^{\ast}$ is squarefree and coprime to $Q$.

Let $N \geq 1$ be large. Then for any $1 \leq T \leq N/(10(\deg{Q})^{\omega(Q)+1})$ there is a choice of $H\in [T, T (\deg{Q})^{\omega(Q)}]$ such that
$$
\max_{G_0 \in \mc{M}_{\leq N}} \left|\sum_{G \in I_H(G_0)} f(G)\right| \gg_Q H^{\frac{1}{2}(\omega(Q)-1)}.
$$
\end{prop}

\begin{rem}
As we shall see, the assumption that $Q/Q^{\ast}$ be squarefree and coprime to $Q$ is satisfied in our application.
\end{rem}

The proof is based on a careful analysis of \emph{Ramanujan sums}. For $G\in \mathcal{M},H \in \mb{F}_q[t]$, the Ramanujan sum $c_G(H)$ is defined by
$$
c_G(H) := \asum_{A \pmod{G}} e_{\mb{F}}\left(\frac{AH}{G}\right),
$$
where $*$ in the sum denotes summation over invertible residue classes. Ramanujan sums satisfy the relation 
\begin{align}\label{eqq18}
\sum_{D|G} c_D(H) = q^{\deg{G}} 1_{G|H},
\end{align}
so that by M\"{o}bius inversion we get
\begin{equation}\label{eq_ramsum_exp}
c_G(H) = \sum_{E|G} \mu(G/E) q^{\deg{E}} 1_{E|H}.
\end{equation}
\begin{lem} \label{lem_ramsumff}
Let $Q \in \mc{M}$, $\deg{Q}\geq 1$, and let $n \in \mb{Z}$. Then
$$
\sum_{\deg{M} < n} c_Q(M) = \begin{cases} \phi(Q) &\text{ if $n \leq 0$} \\ 0 &\text{ if $n \geq 1$.}\end{cases}
$$
\end{lem}
\begin{proof}
Let $S(n)$ denote the sum on the left-hand side. Then
\begin{align*}
S(n) = \sum_{\deg{M} < n} c_Q(M) = c_Q(0) + \sum_{0 \leq m < n} \sum_{\deg{M} = m} c_Q(M),
\end{align*}
where the sum on the right-hand side is interpreted as empty unless $n \geq 1$. If $n \leq 0$, we are done since $c_Q(0) = \phi(Q)$, so suppose $n \geq 1$. 

Expanding $c_Q(M)$ using~\eqref{eq_ramsum_exp}, we get
\begin{align*}
S(n) &= c_Q(0) + \sum_{E|Q} \mu(Q/E)q^{\deg{E}} \sum_{0 \leq m < n} \sum_{\substack{\deg{M} = m \\ E|M}} 1 \\
&= c_Q(0) + \sum_{E|Q} \mu(Q/E) q^{\deg{E}} \sum_{0 \leq m < n} (q-1)q^{m-\deg{E}} 1_{m \geq \deg{E}} \\
&= c_Q(0) + (q-1) \sum_{E|Q} \mu(Q/E) \sum_{\deg{E} \leq m < n} q^m.
\end{align*}
Summing the geometric series, using $\sum_{E|Q}\mu(Q/E) = 0$ for $\deg{Q}\geq 1$ and~\eqref{eq_ramsum_exp}, we get
\begin{align*}
S(n) = c_Q(0) + \sum_{E|Q}\mu(Q/E) (q^{n} - q^{\deg{E}}) = c_Q(0) - \sum_{E|Q}\mu(Q/E) q^{\deg{E}} = c_Q(0) - c_Q(0) = 0.
\end{align*}
This completes the proof of the claim.
\end{proof}

\subsubsection{Proof of Proposition~\ref{prop_lrg_ssd}}
Since $\chi\pmod Q$ is non-principal, we have  $\deg{Q} \geq 1$. Write $Q = P_1^{r_1}\cdots P_k^{r_k}$ where the $P_j$ are all distinct, and set $d_j := \deg{P_j}$ for all $j$. Suppose $f:\mc{M} \to S^1$ is completely multiplicative, with $f(P) = \chi(P)$ for all $P \neq P_j$. We put $H = Tr_1d_1\cdots r_kd_k$ and observe that we have the inequalities
$$
T \leq H \leq T (\max_{1 \leq j \leq k} r_jd_j)^{\omega(Q)} \leq T (\deg{Q})^{\omega(Q)},
$$
as required.

Let $N > 10T(\deg{Q})^{\omega(Q)+1} \geq 10H\deg{Q}$. Then
\begin{align}
\max_{G_0 \in \mc{M}_{\leq N}} \left| \sum_{G \in I_{H}(G_0)} f(G)\right|^2 &\geq q^{-N} \sum_{G_0 \in \mc{M}_N} \left|\sum_{\deg{M} < H} f(G_0+M) \right|^2 \nonumber\\
&= q^{-N} \sum_{G_0 \in \mc{M}_N} \left|\sum_{\text{rad}(D)|Q} f(D) \sum_{\deg{M} < H} \chi\left(\frac{G_0+M}{D}\right)1_{D|(G_0+M)} \right|^2 =: \mc{T}. \label{eq:mcTIntro}
\end{align}
We will show the following lower bound for $\mc{T}$.
\begin{lem}\label{lem:mainModChar}
Assume the hypotheses of Proposition~\ref{prop_lrg_ssd}, and write $Q_S := Q/Q^{\ast}$. Then
$$
\mc{T} \geq \frac{\phi(Q)}{q^{\deg{Q}}} \left(\frac{\phi(Q_S)}{q^{\deg{Q_S}}}\right)^2\prod_{P|Q_S}\left|1-f\bar{\chi^{\ast}}(P)q^{-\deg{P}}\right|^{-2} \left(q^H \sum_{\substack{\textnormal{rad}(D)|Q \\ \deg{D}\geq H}} q^{-\deg{D}}\right) + o_{N \ra \infty}(1).
$$
\end{lem}
\begin{proof}[Deduction of Proposition~\ref{prop_lrg_ssd} assuming Lemma~\ref{lem:mainModChar}]
Note that the product 
$$
\prod_{P|Q_S}|1-f\bar{\chi^{\ast}}(P)q^{-\deg{P}}|^{-2}
$$
is non-vanishing, and therefore $\gg_Q 1$. From Lemma~\ref{lem:mainModChar}, we thus obtain
\begin{align*}
\mc{T} &\gg_Q q^H\sum_{\substack{\text{rad}(D) |Q \\ \deg{D} \geq H}} q^{-\deg{D}} + o_{N \ra \infty}(1)
\geq \sum_{\substack{\text{rad}(D)|Q \\ \deg{D} = H}} 1 + o_{N \to \infty}(1) \\
&= |\{\mbf{\alpha} \in \mb{N}_0^k : \alpha_1r_1d_1 + \cdots + \alpha_kr_kd_k = H\}| + o_{N \to \infty}(1),
\end{align*}
since $Q = P_1^{r_1}\cdots P_k^{r_k}$ with $d_j = \text{deg}(P_j)$, for each $j$. As $H = Tr_1d_1 \cdots r_k d_k$,
\begin{align*}
&|\{\mbf{\alpha} \in \mb{N}_0^k : \alpha_1r_1d_1 + \cdots + \alpha_kr_kd_k = H\}| \\
&\geq |\{\mbf{\alpha} \in \mb{N}_0^k : \alpha_1r_1d_1+\cdots+\alpha_kr_kd_k = Tr_1d_1\cdots r_kd_k \text{ and } \prod_{\substack{1 \leq i \leq k \\ i \neq j}} r_id_i |\alpha_j \text{ for all } 1\leq j \leq k\}| \\
&= |\{\mbf{\beta} \in \mb{N}_0^k : \beta_1+ \cdots + \beta_k = T\}| = \binom{T+k-1}{k-1} \gg_k T^{k-1} \gg_Q H^{k-1}.
\end{align*}
In particular, we obtain $\mc{T} \gg_Q H^{\omega(Q)-1}$. It thus follows from~\eqref{eq:mcTIntro} that
\begin{align*}
\max_{G_0 \in \mc{M}_{\leq N}} \left| \sum_{G \in I_{H}(G_0)} f(G)\right|^2 \geq \mc{T} \gg_Q H^{\omega(Q)-1},
\end{align*}
as required.
\end{proof}

It therefore remains to prove Lemma~\ref{lem:mainModChar}. 
\begin{proof}[Proof of Lemma~\ref{lem:mainModChar}]
The proof of the lemma is a technical computation, but can be divided into several steps. 

\textbf{Step 1: Reduction to a sum over a hyperplane $\pmod{Q}$.} Note that if $\deg{D} > N$ then the only solution to $D|(G_0+M)$ requires $G_0 + M = 0$, which is impossible since $\deg{M} = H < N = \deg{G_0}$. Thus we may additionally assume that $\deg{D} \leq N$ in the inner sum defining $\mc{T}$. 
Splitting into residue classes modulo $Q$ and then expanding the square (and making the change of variables $M \mapsto -M$ for later convenience), we have
\begin{align*}
\mc{T} &= q^{-N} \sum_{G_0 \in \mc{M}_N} \left|\sum_{\substack{\text{rad}(D)|Q \\ \deg{D} \leq N}} f(D) \sum_{\deg{M} < H} \chi\left(\frac{G_0-M}{D}\right)1_{D|(G_0-M)}\right|^2 \\
&= q^{-N} \sum_{G_0 \in \mc{M}_N} \left|\,\,\asum_{A \pmod{Q}} \chi(A) \sum_{\substack{\text{rad}(D)|Q \\ \deg{D} \leq N}} f(D) \sum_{\deg{M} < H} 1_{D|(G_0-M)}1_{(G_0-M)/D \equiv A \pmod{Q}} \right|^2 \\
&= q^{-N} \asum_{A_1,A_2 \pmod{Q}} \chi(A_1) \bar{\chi}(A_2) \sum_{\substack{\text{rad}(D_j)|Q \\ \deg{D_j} \leq N \\ j = 1,2}} f(D_1)\bar{f}(D_2) \\
& \cdot\sum_{\substack{\deg{M_j} < H \\ j = 1,2}} |\{G_0 \in \mc{M}_N : G_0 \equiv M_j \pmod{D_j} , (G_0-M_j)/D_j \equiv A_j \pmod{Q}, j = 1,2\}|.
\end{align*}

Fix momentarily $D_1,D_2$ with $\text{rad}(D_j) |Q$ for $j = 1,2$, and set $D := (D_1,D_2)$ and $\widetilde{D}_j:= D_j/D$ for $j = 1,2$. Note that the pair of congruences $G_0 \equiv M_j \pmod{D_j}$ for $j = 1,2$ is solvable if and only if $D|(M_2-M_1)$, and provided $\deg{[D_1,D_2]} \leq N$ the general solution has the form 
\begin{align*}
G_0 &= R[D_1,D_2] + \frac{M_1L_2D_2 + M_2L_1D_1}{D}\\
&= R[D_1,D_2] + M_1 + \frac{M_2-M_1}{D}L_1D_1 = R[D_1,D_2] + M_2 - \frac{M_2-M_1}{D}L_2D_2,
\end{align*}
where $L_1,L_2$ are reduced residue classes modulo $[D_1,D_2]$ that satisfy $L_1D_1 + L_2D_2 = D$. Thus, provided that $\deg{[D_1,D_2]} \leq N$, we have
\begin{align*}
&|\{G_0 \in \mc{M}_N : G_0 \equiv M_j \pmod{D_j}, (G_0 - M_j)/D_j \equiv A_j \pmod{Q}, j = 1,2\}| \\
&= \left|\left\{R \in \mc{M}_{N-\deg{[D_1,D_2]}} : \begin{cases} R\widetilde{D}_2 + L_1(M_2-M_1)/D &\equiv A_1 \pmod{Q} \\ R\widetilde{D}_1 - L_2(M_2-M_1)/D &\equiv A_2 \pmod{Q} \end{cases}\right\}\right|,
\end{align*}
We note that even the condition $\deg{[D_1,D_2]Q} < N$ may be assumed in what follows, since the contribution to $\mc{T}$ from those $D_1,D_2$ that lack this is
\begin{align*}
&\ll_Q q^{2H-N} \max_{\deg{M_2} < H} \sum_{\substack{\text{rad}(D_j) | Q \\ \deg{D_1} \leq \deg{D_2} \leq N \\ \deg{[D_1,D_2]Q} \geq N}} |\{G_0 \in \mc{M}_N : G_0 \equiv M_2 \pmod{D_2}\}| \\
&\ll q^{2H} \sum_{\substack{\text{rad}(D_2)|Q \\ \deg{D_2} \geq (N-\deg{Q})/2}} q^{-\deg{D_2}} \sum_{\substack{\textnormal{rad}(D_1)\mid Q\\\deg{D_1}\leq N}}1\\
&\ll_Q q^{2H-N/2} N^{O_Q(1)} = o_{N \to \infty}(1),
\end{align*}
as $2H \leq 2H\deg{Q} < N/5$.  

Earlier we had deduced that $D|(M_2-M_1)$. 
Making the change of variables $M'D = M_2-M_1$, we get
\begin{align*}
\mc{T} &= q^{H-N} \sum_{\substack{\text{rad}(D_j)|Q \\ \deg{Q[D_1,D_2]} < N}} f(D_1)\bar{f}(D_2)\,\, \asum_{A_1,A_2 \pmod{Q}} \chi(A_1)\bar{\chi}(A_2) \\
&\cdot \sum_{\deg{M'} < H-\deg{D}} \left|\left\{R \in \mc{M}_{N-\deg{[D_1,D_2]}} : \begin{cases} R\widetilde{D}_2 + M'L_1 &\equiv A_1 \pmod{Q} \\ R\widetilde{D}_1 - M'L_2 &\equiv A_2 \pmod{Q} \end{cases} \right\}\right| + o_{N \to \infty}(1);
\end{align*}
note that if $\deg{D} \geq H$ the summation contains the choice $M'= 0$ alone. It is easy to verify that the system of congruences 
$$
\begin{cases} R\widetilde{D}_2 + M'L_1 &\equiv A_1 \pmod{Q} \\ R\widetilde{D}_1-M'L_2 &\equiv A_2 \pmod{Q} \end{cases}
$$
is solvable if, and only if, 
$$
\begin{cases} M' &\equiv A_1\widetilde{D}_1 - A_2\widetilde{D}_2 \pmod{Q} \\ R &\equiv L_1A_2 + L_2A_1 \pmod{Q}. \end{cases}
$$ 
Therefore, $\mathcal{T}$ is, up to $o_{N\to \infty}(1)$ error, equal to 
\begin{align*}
&q^{H-N} \sum_{\substack{\text{rad}(D_j)|Q \\ \deg{Q[D_1,D_2]} < N}} f(D_1)\bar{f}(D_2)\,\, \asum_{\substack{A_1\pmod{Q} \\ A_2 \pmod{Q}}} \chi(A_1)\bar{\chi}(A_2) \sum_{\substack{\deg{M'} < H-\deg{D} \\ M' \equiv A_1\widetilde{D}_1-A_2\widetilde{D}_2 \pmod{Q}}}\\
&\quad \quad \sum_{\substack{R \in \mc{M}_{N-\deg{[D_1,D_2]}} \\ R \equiv A_1 L_2 + A_2L_1 \pmod{Q}}} 1\\
&= q^{H} \sum_{\substack{\text{rad}(D_j)|Q \\ \deg{Q[D_1,D_2]} < N}} \frac{f(D_1)\bar{f}(D_2)}{q^{\deg{Q[D_1,D_2]}}}\,\, \asum_{\substack{A_1\pmod{Q} \\ A_2 \pmod{Q}}} \chi(A_1)\bar{\chi}(A_2) \sum_{\substack{\deg{M'} < H - \deg{D} \\ M' \equiv A_1\widetilde{D}_1-A_2\widetilde{D}_2 \pmod{Q}}} 1 .
\end{align*}
Changing variables as $D_1 = D\widetilde{D}_1$ and $D_2 = D\widetilde{D}_2$, and reinstating triples $D,\widetilde{D}_1,\widetilde{D}_2$ with $\deg{Q[D_1,D_2]} = \deg{QD\widetilde{D}_1\widetilde{D}_2} \geq N$, this is equal to

\begin{align*}
&= q^{H} \sum_{\text{rad}(D)|Q} q^{-\deg{D}} \sum_{\substack{\text{rad}(\widetilde{D}_j)|Q \\ (\widetilde{D}_1,\widetilde{D}_2) = 1}} \frac{f(\widetilde{D}_1)\bar{f}(\widetilde{D}_2)}{q^{\deg{Q\widetilde{D}_1\widetilde{D}_2}}} \asum_{\substack{A_1\pmod{Q} \\ A_2 \pmod{Q}}} \chi(A_1)\bar{\chi}(A_2) \sum_{\substack{\deg{M'} < H - \deg{D} \\ M' \equiv A_1\widetilde{D}_1-A_2\widetilde{D}_2 \pmod{Q}}} 1 \\
&+ O(q^{2H-N/3+o_Q(1)}),
\end{align*}
the error term being $o_{N \to \infty}(1)$ since $N > 7H$. 

\textbf{Step 2: Decoupling $\widetilde{D}_1$ and $\widetilde{D}_2$ via Ramanujan sums.}
 By~\eqref{eqq18} and the fact that $\textnormal{rad}(\widetilde{D_j})\mid Q$, inserting additive characters $\pmod{Q}$ to detect the condition $M' \equiv A_1 \widetilde{D}_1 - A_2 \widetilde{D}_2 \pmod{Q}$ yields
\begin{align*}
&\asum_{A_1,A_2 \pmod{Q}} \chi(A_1)\bar{\chi}(A_2) \sum_{\substack{\deg{M'} < H-\deg{D} \\ M' \equiv A_1\widetilde{D}_1-A_2\widetilde{D}_2 \pmod{Q}}} 1 \\
&= q^{-\deg{Q}} \sum_{C \pmod{Q}}\,\, \sum_{\deg{M'} < H-\deg{D}} e_{\mb{F}}\left(\frac{CM'}{Q}\right) \tau(\chi,-C\widetilde{D}_1) \bar{\tau}(\chi,-C\widetilde{D}_2).
\end{align*}
Write $Q = Q^{\ast} Q_S$, where $Q^{\ast}$ is the conductor of $\chi$; by assumption, we have $(Q^{\ast},Q_S) = 1$ and $Q_S$ squarefree. By Lemma~\ref{lem_imprim}, for each $j = 1,2$ we have
\begin{align*}
\tau(\chi,-C\widetilde{D}_j) = \tau(\chi^{\ast}) \phi((Q_S,C\widetilde{D}_j)) \mu(Q_S/(Q_S,C\widetilde{D}_j)) \chi^{\ast}(-Q_S\bar{C\widetilde{D}_j})1_{(C\widetilde{D}_j,Q^{\ast}) = 1}.
\end{align*}
We insert these expressions into the above, using $|\tau(\chi^{\ast})|^2 = q^{\deg{Q^{\ast}}}$. Removing the condition $(\widetilde{D}_1,\widetilde{D}_2) = 1$ by M\"{o}bius inversion and splitting the products $C\widetilde{D}_j$ according to $(C\widetilde{D}_j,Q_S)$, we obtain
\begin{align*}
\mc{T} &= q^{H-\deg{QQ_S}}\sum_{\text{rad}(D)|Q} q^{-\deg{D}} \sum_{E_1,E_2|Q_S} \phi(E_1)\mu\left(\frac{Q_S}{E_1}\right)\phi(E_2)\mu\left(\frac{Q_S}{E_2}\right)  \\
&\cdot \sum_{\substack{\text{rad}(\widetilde{D}_j)|Q_S \\ (\widetilde{D}_1,\widetilde{D}_2) = 1}} \frac{f\bar{\chi^{\ast}}(\widetilde{D}_1)\bar{f}\chi^{\ast}(\widetilde{D}_2)}{q^{\deg{\widetilde{D}_1\widetilde{D}_2}}}\sum_{\deg{M'} < H- \deg{D}}\,\, \sum_{\substack{C \pmod{Q} \\ (C,Q^{\ast}) = 1 \\  E_j = (Q_S,C\widetilde{D}_j), j = 1,2}} e_{\mb{F}}(CM'/Q) + o_{N \ra \infty}(1) \\
&= q^{H-\deg{QQ_S}}\sum_{\text{rad}(D)|Q} q^{-\deg{D}} \sum_{E_j|Q_S} \phi(E_1)\mu\left(\frac{Q_S}{E_1}\right)\phi(E_2)\mu\left(\frac{Q_S}{E_2}\right) \sum_{L|Q_S} \frac{\mu(L)}{q^{2\deg{L}}}  \\
&\cdot \sum_{\substack{\text{rad}(\widetilde{D}_j)|Q_S}} \frac{f\bar{\chi^{\ast}}(\widetilde{D}_1)\bar{f}\chi^{\ast}(\widetilde{D}_2)}{q^{\deg{\widetilde{D}_1\widetilde{D}_2}}}\sum_{\deg{M'} < H- \deg{D}}\,\, \sum_{\substack{C \pmod{Q} \\ (C,Q^{\ast}) = 1 \\  E_j = (Q_S,CL\widetilde{D}_j), j = 1,2}} e_{\mb{F}}(CM'/Q) + o_{N \ra \infty}(1).
\end{align*}

We next define $F := (C,Q_S)$ for each $C$ modulo $Q$. We also decompose $\widetilde{D}_j = D_j'D_j''$, where $\text{rad}(D_j')|[F,L]$ and $\text{rad}(D_j'')|Q_S/[F,L]$, so that $E_j = \text{rad}(D_j'')[F,L]$ for each $j = 1,2$. This leads to the expression 
\begin{align*}
\mc{T} &= q^{H-\deg{QQ_S}} \sum_{\text{rad}(D)|Q}q^{-\deg{D}} \sum_{F | Q_S} \sum_{L|Q_S} \frac{\mu(L)}{q^{2\deg{L}}} \phi([F,L])^2 \sum_{\substack{\text{rad}(D_j')|[F,L] \\ j = 1,2}} \frac{f\bar{\chi^{\ast}}(D_1')\bar{f}\chi^{\ast}(D_2')}{q^{\deg{D_1'D_2'}}} \\
&\cdot \sum_{\substack{\text{rad}(D_j'')|Q_S/[F,L] \\ j = 1,2}} \mu\left(\frac{(Q_S/[F,L])}{\text{rad}(D_1'')}\right)\phi(\text{rad}(D_1''))\mu\left(\frac{(Q_S/[F,L])}{\text{rad}(D_2'')}\right)\phi(\text{rad}(D_2''))\frac{f\bar{\chi^{\ast}}(D_1'')\bar{f}\chi^{\ast}(D_2'')}{q^{\deg{D_1''D_2''}}} \\
&\cdot \sum_{\deg{M'} < H-\deg{D}}\,\, \sum_{\substack{C \pmod{Q} \\ (C,Q^{\ast}) = 1 \\ F = (C,Q_S)}} e_{\mb{F}}\left(\frac{CM'}{Q^{\ast}Q_S}\right).
\end{align*}
Replacing $C$ by $\widetilde{C} := C/F$ in the innermost sum, and noting that $(C/F,Q/F) = 1$ in that case, it follows from Lemma~\ref{lem_ramsumff} (and $\deg{Q/F} \geq \deg{Q^{\ast}} \geq 1$ since $\chi$ is non-principal) that
\begin{align*}
\sum_{\deg{M'} < H-\deg{D}}\,\, \sum_{\substack{C \pmod{Q} \\ (C,Q^{\ast}) = 1 \\ F = (C,Q_S)}} e_{\mb{F}}(CM'/Q) &= \sum_{\deg{M'} < H-\deg{D}}\,\,\, \asum_{\substack{\widetilde{C} \pmod{Q/F}}}e_{\mb{F}}(\widetilde{C}M'/(Q/F)) \\
&= \sum_{\deg{M'} < H-\deg{D}} c_{Q/F}(M') = \phi(Q/F)1_{\deg{D} \geq H},
\end{align*}
for each $F|Q_S$. Inserting this into the expression for $\mc{T}$ then gives
\begin{align*}
\mc{T} &= \left(q^H\sum_{\substack{\text{rad}(D)|Q \\ \deg{D} \geq H}} q^{-\deg{D}}\right) \cdot \frac{\phi(Q^{\ast})}{q^{\deg{QQ_S}}} \sum_{F|Q_S} \phi\left(\frac{Q_S}{F}\right) \sum_{L|Q_S} \frac{\mu(L)\phi([F,L])^2 }{q^{2\deg{L}}}\\
&\cdot \sum_{\substack{\text{rad}(D_j')|[F,L] \\ j =1,2}} \frac{f\bar{\chi^{\ast}}(D_1')\bar{f}\chi^{\ast}(D_2')}{q^{\deg{D_1'D_2'}}}  \sum_{\substack{\text{rad}(D_j'')|\frac{Q_S}{[F,L]} \\ j = 1,2}} \prod_{j=1}^2 \mu\left(\frac{\frac{Q_S}{[F,L]}}{\text{rad}(D_j'')}\right)\phi(\text{rad}(D_j''))\cdot \frac{f\bar{\chi^{\ast}}(D_1'')\bar{f}\chi^{\ast}(D_2'')}{q^{\deg{D_1''D_2''}}}\\
&+ o_{N \ra \infty}(1).
\end{align*}
\textbf{Step 3: Concluding the proof.} Finally, we make one last change of variable $G := [F,L]$. For each $G|Q_S$, we have (using the squarefreeness of $Q_S$ repeatedly)
\begin{align*}
&\sum_{\substack{F,L|Q_S \\ [F,L] = G}} \phi(Q_S/F) \frac{\mu(L)}{q^{2\deg{L}}} = \sum_{R|G} \frac{\mu(R)}{q^{2\deg{R}}} \sum_{L'F' = G/R} \phi\left(\frac{Q_S/R}{F'}\right) \frac{\mu(L')}{q^{2\deg{L'}}} \\
&= \phi\left(\frac{Q_S}{G}\right) \sum_{R|G} \frac{\mu(R)}{q^{2\deg{R}}} \sum_{F'L' = G/R} \phi\left(\frac{G/R}{F'}\right) \frac{\mu(L')}{q^{2\deg{L'}}} = \phi\left(\frac{Q_S}{G}\right) \sum_{R|G} \frac{\mu(R)}{q^{2\deg{R}}} \sum_{L'|G/R} \frac{\phi(L')\mu(L')}{q^{2\deg{L'}}} \\
&= \phi\left(\frac{Q_S}{G}\right) \sum_{R|G} \frac{\mu(R)}{q^{2\deg{R}}} \prod_{P|G/R} \left(1-q^{-\deg{P}}(1-q^{-\deg{P}})\right) = \phi\left(\frac{Q_S}{G}\right)\prod_{P|G} \left(1-q^{-\deg{P}}\right)\\
&= q^{-\deg{G}}\phi(Q_S).
\end{align*}
Applying the change of variables and the above identity into the previous expression for $\mc{T}$, we obtain
\begin{align*}
\mc{T} &= \left(q^H \sum_{\substack{\text{rad}(D)|Q \\ \deg{D} \geq H}} q^{-\deg{D}}\right) \frac{\phi(Q^{\ast})\phi(Q_S)}{q^{\deg{QQ_S}}} \sum_{G|Q_S} \frac{\phi(G)^2}{q^{\deg{G}}} \sum_{\substack{\text{rad}(D_j')|G \\ j = 1,2}} \frac{f\bar{\chi^{\ast}}(D_1')\bar{f}\chi^{\ast}(D_2')}{q^{\deg{D_1'D_2'}}} \\
&\cdot \sum_{\substack{\text{rad}(D_j'')|Q_S/G \\ j = 1,2}} \mu\left(\frac{Q_S/G}{\text{rad}(D_1'')}\right)\phi(\text{rad}(D_1''))\mu\left(\frac{Q_S/G}{\text{rad}(D_2'')}\right)\phi(\text{rad}(D_2''))\frac{f\bar{\chi^{\ast}}(D_1'')\bar{f}\chi^{\ast}(D_2'')}{q^{\deg{D_1''D_2''}}} + o_{N \ra \infty}(1)\\
&= \frac{\phi(Q)}{q^{\deg{QQ_S}}}\sum_{G|Q_S} \frac{\phi(G)^2}{q^{\deg{G}}} \left|\sum_{\text{rad}(D')|G} \frac{f\bar{\chi^{\ast}}(D')}{q^{\deg{D'}}}\right|^2 \left|\sum_{\text{rad}(D'')|Q_S/G} \mu\left(\frac{Q_S/G}{\text{rad}(D'')}\right)\phi(\text{rad}(D'')) \frac{f\bar{\chi^{\ast}}(D'')}{q^{\deg{D''}}}\right|^2 \\
&\cdot \left(q^{H} \sum_{\substack{\text{rad}(D)|Q \\ \deg{D} \geq H}} q^{-\deg{D}}\right) + o_{N \ra \infty}(1) \\
&\geq \frac{\phi(Q)}{q^{\deg{Q}}} \left(\frac{\phi(Q_S)}{q^{\deg{Q_S}}}\right)^2\prod_{P|Q_S}\left|1-f\bar{\chi^{\ast}}(P)q^{-\deg{P}}\right|^{-2} \cdot \left(q^H\sum_{\substack{\text{rad}(D) |Q \\ \deg{D} \geq H}} q^{-\deg{D}}\right) + o_{N \ra \infty}(1),
\end{align*}
where in the last step we used positivity to bound the sum over $G$ from below by the term at $G = Q_S$, and the factorization
\begin{align*}
\sum_{\text{rad}(D')|Q_S} \frac{f\bar{\chi^{\ast}}(D)}{q^{\deg{D}}} = \prod_{P|Q_S} \left(1-f\bar{\chi^{\ast}}(P)q^{-\deg{P}}\right)^{-1}.
\end{align*}
This completes the proof.
\end{proof}

\subsubsection{Modified characters to prime power modulus}

\begin{proof}[Proof of Theorem~\ref{thm_edpff}]
($\Rightarrow$) Suppose $f: \mc{M} \to S^1$ is a completely multiplicative function for which $\mc{S}_f < \infty$. By Proposition~\ref{prop_redtoDirPret}, there is a primitive Dirichlet character $\chi$ modulo $Q'$, a primitive short interval character $\xi$ of length $\nu \geq 0$ and $\theta \in [0,1]$ such that $\mb{D}(f,\chi \xi e_{\theta};\infty) < \infty$. 

We start with the case $Q'=1$. Let $N$ be large and $1\leq H\leq N-\nu-1$. Set $f_1(G) := fe_{-\theta}\bar{\xi}(G)$ for each $G \in \mc{M}$, so that $\mb{D}(f_1,1;\infty) < \infty$. Further, note that $\xi e_{\theta}$ is constant on intervals $I_H(G_0)$ for all $G_0 \in \mc{M}_N$. We thus obtain
\begin{align*}\begin{split}
\max_{G_0 \in \mc{M}_N} \left|\sum_{\deg{M} < H} f(G_0+M)\right|& = \max_{G_0 \in \mc{M}_N} \left|\sum_{\deg{M} < H} f_1(G_0+M)\right|\\ &\geq q^{-N} \sum_{G_0 \in \mc{M}_N}\left|\sum_{\deg{M} < H} f_1(G_0+M)\right|\\
&\geq q^{-N} \left|\sum_{G_0 \in \mc{M}_N}\sum_{\deg{M} < H} f_1(G_0+M)\right|\\
&=q^{H-N}\left|\sum_{G \in \mc{M}_N} f_1(G)\right|, 
\end{split}
\end{align*}
where we used the triangle inequality and the fact that $M + \mc{M}_N = \mc{M}_N$ for all $\deg{M} < H$. 
We now apply Delange's theorem in function fields (see~\cite[Theorem 1.4.1]{KluThe}) to $f_1$, which gives that
$$
q^{H-N}\left|\sum_{G \in \mc{M}_N} f_1(G)\right| = (c+o_{N \ra \infty}(1))q^H,
$$
where, since $f_1$ is 1-pretentious, we have
\begin{align*}
c=\prod_{P\in \mc{P}}(1-q^{-\deg{P}})(1-f_1(P)q^{-\deg{P}})^{-1}\neq 0.    
\end{align*}
It follows directly that $S_f = \infty$, a contradiction.

We are left with the case $Q'\neq 1$, so $\deg{Q'}\geq 1$. We apply Proposition~\ref{prop_redtofin} to $f$ to deduce that
$$
S := \{P : f(P) \neq \chi(P)\xi(P)e_{\theta}(P)\}
$$
is finite. Put $Q := [Q',\prod_{P \in S} P] = Q'Q''$, where $Q''$ is squarefree and coprime to $Q'$. 

Then $f\bar{\xi}e_{-\theta}$ is a modified character modulo $Q$ (as per Remark~\ref{rem_genchar}), and $\chi$ is non-principal with conductor $Q'$. By Proposition~\ref{prop_lrg_ssd} (applied with $T$ being a large constant, so that $H$ is small compared to $N$ and $\xi e_{\theta}$ is constant on $I_H(G_0)$ for any $G_0 \in \mc{M}_N$) we find that $\omega(Q) = 1$. Thus, $Q = Q'$ and $\omega(Q') = 1$, so $\chi$ is a primitive Dirichlet character modulo a prime power. 

To conclude, we thus have $f(G) = \widetilde{\chi}(G) \xi(G) e_{\theta}(G)$ for all $G$, where $\chi$ is a primitive Dirichlet character modulo $Q = P^r$ for some $r \geq 1$ and some prime $P$, $\widetilde{\chi}$ is a modified character corresponding to $\chi$, and $\xi$ has bounded length.

($\Leftarrow$) Conversely, let $f(G) = \widetilde{\chi}(G) \xi(G) e_{\theta}(G)$ for all $G$, where $\chi$ is a primitive Dirichlet character modulo $Q = P^r$ for some $r \geq 1$ and some prime $P$, $\widetilde{\chi}$ is a modified character corresponding to $\chi$, and $\xi$ has length $\nu$. Denote $f_1(G):=\widetilde{\chi}(G)$. As we noted before, $\xi e_{\theta}$ is constant on $I_H(G_0)$ for any $G_0 \in \mc{M}_N$ and $H\leq N-\nu-1$, so  $S_{f}=S_{\widetilde{\chi}}$. Thus it suffices to show that $S_{\widetilde{\chi}}<\infty$. 

Let $H \geq 1$ and suppose $N \geq H$. For any $G_0 \in \mc{M}_{N}$,
\begin{align} 
\sum_{G \in I_H(G_0)} \widetilde{\chi}(G) &= \sum_{k \geq 0} f(P)^{k} \sum_{\deg{M} < H} \chi((G_0-M)/P^k) \nonumber\\
&= \asum_{A \pmod{P^r}} \chi(A) \sum_{k \geq 0} f(P)^k \sum_{\substack{\deg{M} < H \\ M \equiv G_0 \pmod{P^k} \\ (G_0-M)/P^k \equiv A \pmod {P^r}}} 1. \label{eq_ssdomega1}
\end{align}
Consider first the contribution from $k<H/\deg{P}-r$. Making the change of variables $M=B_k(G_0)+P^kM'$ in the inner sum over $M$, where $B_k(G_0)$ is the residue class of $G_0$ mod $P^k$, we see that
\begin{align*}
 \sum_{\substack{\deg{M} < H \\ M \equiv G_0 \pmod{P^k} \\ (G_0-M)/P^k \equiv A \pmod {P^r}}} 1&= \sum_{\substack{\deg{M'}<H-k\deg{P}\\M'\equiv A\pmod{P^r}}}1=\sum_{\substack{\deg{M'}<H-k\deg{P}\\M'\equiv 0\pmod{P^r}}}1,
\end{align*}
which is independent of $A$. Thus, by orthogonality these values of $k$ contribute nothing to~\eqref{eq_ssdomega1}. 

In the range $k > \frac{H}{\deg{P}}$, there is at most one polynomial $M$ that contributes for at most one such value of $k$ (and in this case, $M$ must represent the projection of $G_0$ to $\text{span}_{\mb{F}_q}\{1,\ldots,t^H\}$). This results in a $O(1)$ term. 

It follows that
\begin{align*}
\sum_{G \in I_H(G_0)} f(G) =  \sum_{H/\deg{P}-r \leq k \leq H/\deg{P}} f(P)^k \sum_{\substack{\deg{M} < H \\ M \equiv G_0 \pmod{P^k}}} \chi((G_0-M)/P^k) + O_{r}(1),
\end{align*} 
and estimating each term by the triangle inequality this is $\ll 1$, uniformly over $H$. It follows that 
$$
\limsup_{N \ra \infty} \max_{G_0 \in \mc{M}_N} \left|\sum_{G \in I_H(G_0)} \widetilde{\chi}(G)\right| \ll 1
$$
uniformly over $H$, and hence $S_{\widetilde{\chi}} < \infty$, as claimed.
\end{proof}

\begin{proof}[Proof of Corollary~\ref{short_extr}]
The proof of Corollary~\ref{short_extr} is identical to that of Theorem~\ref{thm_edpff}, save that by the second conclusion in Proposition~\ref{prop_redtoDirPret} we may assume that $\xi$ is quadratic for general $q$ and trivial if $q$ is odd, while $\chi$ is real and $\theta \in \{0,1/2\}$. 
\end{proof}

\section{The Lexicographic Discrepancy}\label{sec_EDP2}

We fix once and for all a lexicographic ordering $\langle \cdot \rangle$ of $\mathbb{F}_q[t]$ (recalling the necessary property that $\langle 0 \rangle = 0$).  Suppose $f: \mc{M} \ra S^1$ is a completely multiplicative function, such that
$$
\sup_{N \geq 1} \left|\sum_{\substack{G \in \mc{M} \\ \langle G \rangle < N}} f(G)\right| < \infty.
$$
We remark that on taking $N = q^n$ for $n \geq 1$, this shows that $\mc{D}_g < \infty$. Taking $N=\langle G_0\rangle$ for any $G_0\in \mathcal{M}_n$, we see that
$$
\sup_{n \geq 1} \sup_{G_0 \in \mc{M}_{n}} \left|\sum_{\substack{G \in \mc{M} \\ \langle G \rangle \leq \langle G_0 \rangle}} f(G)\right| < \infty.
$$
By the triangle inequality, it also follows that for any $h\geq 1$,
$$
\sup_{n \geq 1}  \sup_{G_0 \in \mc{M}_n}\left|\sum_{\substack{G \in \mc{M} \\ \langle G_0\rangle \leq \langle G \rangle < \langle G_0 \rangle+q^h}} f(G) \right| < \infty.
$$
But as pointed out in Section~\ref{sub: proofideas}, the short interval sums $I_h(G_0)$ coincide with the sum in absolute values whenever $n \geq h$ and $t^{h-1}|G_0$. 
Thus, we deduce that $\mc{S}_f < \infty$. By Theorem~\ref{thm_edpff}, we may conclude that $f = \widetilde{\chi}_{\alpha} \xi e_{\theta}$, where $\xi$ is a short interval character of bounded length, $\theta \in [0,1]$ and $\widetilde{\chi}_{\alpha}$ is a primitive modified character with prime power modulus $P^r$, such that 
\begin{align}\label{eq200}
\widetilde{\chi}_{\alpha}(P) = e(\alpha),    
\end{align}
for some $\alpha \in \mb{R}/\mb{Z}$. We will use this notation in the sequel.

We have thus reduced our task to showing the following.  In the sequel we write $\widetilde{\chi} = \widetilde{\chi}_{\alpha}$ for ease of notation.

\begin{prop}\label{prop_lexico}
Let $g:\mathcal{M}\to S^1$ be of the form $g=\widetilde{\chi}\xi e_{\theta}$ with $\widetilde{\chi}$ a primitive modified character associated to a prime power modulus, $\xi$ a short interval character, and $\theta\in \mathbb{R}$. Then we have
\begin{align}\label{eq132}
\sup_{N\geq 1}\left|\sum_{\substack{G\in \mathcal{M}\\\langle G\rangle< N}}g(G)\right|=\infty.     
\end{align}
\end{prop}

Assume for the sake of contradiction that~\eqref{eq132} fails. Before proceeding to the proof of Proposition~\ref{prop_lexico} let us make some observations. 

Firstly, the function $g$ may be extended naturally to all of $\mathbb{F}_q[t]$ by the formula 
$$
g(G)=g(P)^{v_P(G)}\chi(G/P^{v_P(G)})\xi(G)e(\theta \deg{G}),
$$ 
since $\xi$ and $\chi$ are both defined on all of $\mathbb{F}_q[t]$.

Secondly, we may assume that $\sum_{G\in \mathcal{M}_{n}}g(G)$ is bounded, as otherwise 
\begin{align*}
\sum_{\substack{G\in \mc{M}\\\langle G\rangle< q^{n+1}}}g(G)-\sum_{\substack{G\in \mc{M}\\\langle G\rangle< q^{n}}}g(G)=\sum_{G\in \mathcal{M}_{n}}g(G)    
\end{align*}
is unbounded, implying that the claim~\eqref{eq132} holds. 

The next lemma will allow us to study more precisely the behaviour of long interval sums of modified characters, which will be crucial in the proof of Proposition~\ref{prop_lexico}.

\begin{lem} \label{lem_omega1_lsd}
Let $f: \mc{M} \to S^1$ be a fixed completely multiplicative function. Suppose there exist $\theta \in [0,1]$, a short interval character $\xi$ of length $\nu \geq 0$, and a non-principal Dirichlet character modulo $P^r$, where $P\in \mathcal{P}$ and $r \geq 1$, such that $f(P') = \chi(P')\xi(P')e_{\theta}(P')$ for all $P' \neq P$. 
\begin{enumerate}
    \item  For any $H \geq 1$,
$$
\sum_{M \in \mc{M}_{< H}} f(M) = \sum_{M' \in \mc{M}_{< \nu + r\deg{P}}} \chi\xi(M')e_{\theta}(M') \sum_{0 \leq k < (H-\deg{M'})/\deg{P}} f(P)^k.
$$

\item If $f(P)$ is a $d$th root of unity with $d\geq 1$, then $H \mapsto \sum_{M \in \mc{M}_{< H}} f(M)$ is $d\cdot \deg{P}$-periodic.
\end{enumerate}
\end{lem}
\begin{proof}
(1) We have
\begin{align*}
\sum_{M \in \mc{M}_{< H}} f(M) &= \sum_{m < H} \sum_{M \in \mc{M}_m} f(M) \\
&= \sum_{m < H} e(\theta m)\sum_{0 \leq k \leq m/\deg{P}} (fe_{-\theta})(P)^k \sum_{\substack{M \in \mc{M}_m \\ M \equiv 0 \pmod{P^k}}} \chi\xi(M/P^k) \\
&= \sum_{m < H} e(\theta m) \sum_{0 \leq k \leq m/\deg{P}} (fe_{-\theta})(P)^k \sum_{M' \in \mc{M}_{m-k\deg{P}}} \chi\xi(M').
\end{align*}
Swapping orders of summation, this equals to
\begin{align*}
&\sum_{0 \leq k < H/\deg{P}} f(P)^k \sum_{k \deg{P} \leq m < H} e(\theta(m-k\deg{P})) \sum_{M' \in \mc{M}_{m-k\deg{P}}} \chi\xi(M') \\
&= \sum_{0 \leq k < H/\deg{P}} f(P)^k \sum_{0 \leq j < H-k\deg{P}} e(j\theta) \sum_{M' \in \mc{M}_j}\chi\xi(M') \\
&= \sum_{0 \leq j < H} e(j\theta) \left(\sum_{M' \in \mc{M}_j} \chi\xi(M')\right)\sum_{0 \leq k < (H-j)/\deg{P}} f(P)^k.
\end{align*}
If $j \geq \nu + r\deg{P}$ then the sum over $M'$ is 0, as is seen by partitioning $\mc{M}_j$ into short intervals of the form $I_{j-\nu}(G)$ and using the orthogonality of Dirichlet characters. Thus, the above simplifies to
\begin{align}\label{eq133}
&\sum_{0 \leq j < \nu + r\deg{P}} e(j\theta) \sum_{M' \in \mc{M}_j} \chi\xi(M') \sum_{0 \leq k < (H-j)/\deg{P}} f(P)^k\nonumber\\
&= \sum_{M' \in \mc{M}_{< \nu + r\deg{P}}} \chi\xi(M')e_{\theta}(M') \sum_{0 \leq k < (H-\deg{M'})/\deg{P}} f(P)^k. 
\end{align}
This proves the first claim.

(2) If $f(P)\neq 1$, this follows immediately from (1), since $F(n):=\sum_{0\leq k\leq n}f(P)^k$ is $d$-periodic by the fact that the $d$th roots of unity sum up to $0$. If instead $f(P)=1$, then the claim follows by noting that $F(n+1)=F(n)+1$ and using (1) and the orthogonality relations for $\chi\xi$~\cite[Exercise 5.1.2]{EffHay}.
\end{proof}

Let us now introduce some notation. Denote the partial sums of a function $f: \mb{F}_q[t] \ra \mb{C}$ in the lexicographic ordering over monic and non-monic polynomials by
\begin{align*}
S_N^{\mathcal{M}}(f):=\sum_{\substack{G\in \mathcal{M}\\\langle G\rangle< N}}f(G) \quad  \textnormal{and}\quad S_N(f):=\sum_{\substack{G\in \mathbb{F}_q[t]\\\langle G\rangle< N}}f(G).  
\end{align*}
Similarly denote the partial sums arranged according to degree over monic and non-monic polynomials  by
\begin{align*}
\Sigma_{N}^{\mathcal{M}}(f):=\sum_{\substack{G\in \mc{M}\\\deg{G}< N}}f(G) \quad \textnormal{and}\quad \Sigma_{N}(f):=\sum_{\substack{G\in \mathbb{F}_q[t]\\\deg{G}< N}}f(G).    
\end{align*}
We can express the sum $\Sigma_{n_j}(\widetilde{\chi})$ in terms of the corresponding monic sum $\Sigma_{n_j}^{\mc{M}}(\widetilde{\chi})$ as follows.  Since every non-zero polynomial in $\mathbb{F}_q[t]$ can be uniquely written as $cG$ where $G\in \mathcal{M}$ and $c\in \mathbb{F}_q^{\times}$, for all $n\geq 1$ we have
\begin{align}\label{eqq118}
\Sigma_{n}(\widetilde{\chi})= \Sigma_{n}^{\mathcal{M}}(\widetilde{\chi})\sum_{c\in \mathbb{F}_q^{\times}}\widetilde{\chi}(c).
\end{align}
If $\zeta$ is any generator of $\mathbb{F}_q^{\times}$, then
\begin{align}\label{eqq119}
\sum_{c\in \mathbb{F}_q^{\times}}\widetilde{\chi}(c)=\sum_{0\leq j\leq q-2}\widetilde{\chi}(\zeta)^j=(q-1)1_{\widetilde{\chi}(\zeta)=1}:=c_q,  \end{align}
where we used the fact that $x^{q-1}=1$ for all $x\in \mathbb{F}_q^{\times}$.

Our proof of Proposition~\ref{prop_lexico} distinguishes the case $P = t$ from $P \neq t$.
For the case $P=t$ we begin with the following lemma.
\begin{lem} \label{lem:peqtRed}
Suppose $g = \xi e_{\theta} \widetilde{\chi}$, where $\widetilde{\chi}$ is a modified non-principal character modulo $t^r$, such that $S_N^{\mc{M}}(g) = O(1)$ uniformly over all $N \geq 1$. Then $\widetilde{\chi}(t) = 1$.
\end{lem}
\begin{proof}
We observe that if $N,m \geq 1$ and $M < q^m$ then, subject to $M \equiv N \equiv 0 \pmod{q^r}$ we have
\begin{equation} \label{eq:iterqExp}
S_{q^mN + M}(\widetilde{\chi}) = \widetilde{\chi}(t)^{m}S_N(\widetilde{\chi}) + S_M(\widetilde{\chi}).
\end{equation}
To see this, we first decompose
\begin{equation*}
S_{q^mN + M}(\widetilde{\chi}) = \sum_{j \geq 0} \widetilde{\chi}(t)^j \sum_{\lla G' \rra \leq q^{m-j}N + M/q^j} \chi(G').
\end{equation*}

Next, we remark that if $a \geq 0$ is such that $q^a \leq A < q^{a+1}$ then
\begin{align*}
\sum_{\lla G \rra \leq q^r A} \chi(G) &= \sum_{\deg{G} < a + r} \chi(G) + \sum_{q^{r+a} \leq \langle G \rangle \leq q^r A} \chi(G) = \sum_{\lla M \rra \leq (A-q^a)q^r} \chi(t^{r+a}+ M)\\
&= \sum_{\lla M \rra \leq (A-q^a)q^r} \chi(M),
\end{align*}
and so by induction we obtain, for each $j \geq 0$,
$$
\sum_{\lla G' \rra \leq q^{m-j}N + M/q^j} \chi(G') = \sum_{\lla G \rra \leq R_j(N,M)} \chi(G),
$$
where $R_j(M,N) \in \{0,1,\ldots,q^r-1\}$ satisfies $R_j(M,N) \equiv \llf \frac{q^{m} N + M}{q^j} \rrf \pmod{q^r}$. Now, if $m > j$ and $q^r|N$ we have $R_j(M,N) \equiv \llf M/q^j \rrf \pmod{q^r}$, and so
$$
\sum_{0 \leq j < m} \widetilde{\chi}(t)^j \sum_{\lla G \rra \leq M/q^j} \chi(G) = \sum_{\lla G \rra \leq M} \widetilde{\chi}(G) = S_M(\widetilde{\chi}).
$$

Next, suppose $j \geq m$. In this case, $M/q^j < 1$ and $\llf \frac{q^mN + M}{q^j} \rrf = \llf q^{m-j}N\rrf$, since if the floor was one larger this would mean that
$$
1 > \{N/q^{j-m}\} > 1 - \frac{M}{q^j} > 1-1/q^{j-m},
$$
which is impossible. Thus, we have
\begin{align*}
\sum_{j \geq m} \widetilde{\chi}(t)^j \sum_{\lla G' \rra \leq q^{m-j} N + M/q^j} \chi(G') &= \sum_{j \geq m} \widetilde{\chi}(t)^j \sum_{\lla G \rra \leq q^{m-j}N} \chi(G) = \widetilde{\chi}(t)^m \sum_{l \geq 0} \sum_{\ss{\lla G\rra \leq N \\ t^l || G}} \widetilde{\chi}(G)\\
&= \widetilde{\chi}(t)^m S_N(\widetilde{\chi}),
\end{align*}
and~\eqref{eq:iterqExp} follows. 

Now, we iterate~\eqref{eq:iterqExp} as follows. Assume there is $A \equiv 0 \pmod{q^r}$ such that $S_A(\widetilde{\chi}) \neq 0$, and let $K \geq \nu + 1$ be chosen so that $q^K > A$ (with $\nu$ the length of $\xi$). For $J \geq 1,$ let $\{m_j\}_{j \leq J}$ be an increasing sequence of integers for which $|e(m_jK\alpha)-1| < 1/100$ for each $j$. Setting $B := A(1 + q^{m_1K} + \cdots + q^{m_JK})$, we obtain
$$
S_B(\widetilde{\chi}) = S_A(\widetilde{\chi}) + e(m_1K\alpha) S_{(B-A)/q^{m_1K}}(\widetilde{\chi}) = S_A(\widetilde{\chi})\left(1 + e(m_1K\alpha) + \ldots + e(m_JK\alpha)\right).
$$
It follows that if $S_A(\widetilde{\chi}) \neq 0$ then $|S_B(\widetilde{\chi})| \gg J$. We then have
$$
\sum_{\ss{\lla G \rra < B \\ G \in \mc{M}}} g(G) = \sum_{\ss{\lla G \rra < q^{m_JK} \\ G \in \mc{M}}} g(G) + e_{\theta}\xi(t)^{m_J} \left(S_B(\widetilde{\chi}) - S_{q^{m_JK}}(\widetilde{\chi})\right),
$$
and as the left-most two terms are both bounded we obtain that $|S_{q^{m_JK}}(\widetilde{\chi})| \gg J$.  But as $m_JK >r$ can be assured when $J$ is sufficiently large,~\eqref{eq:iterqExp} (with $N = q^r$ and $M = 0$) implies that $|S_{q^r}(\widetilde{\chi})| \gg J$, which is an obvious contradiction as $J \ra \infty$. 

Thus, suppose instead that $S_{q^r N}(\widetilde{\chi}) = 0$ for all $N \geq 1$. In this case, it suffices to notice that then, $$S_{q^r(N+1)}(\widetilde{\chi}) - S_{q^rN}(\widetilde{\chi}) = 0$$ for all $N\geq 1$. Specializing $N_1 = q^{M_1}$ and $N_2 = q^{M_2}$, where $M_1,M_2 \geq r$, we obtain in both cases that
$$
0 = S_{q^r(N_j+1)}(\widetilde{\chi}) - S_{q^rN_j}(\widetilde{\chi})=\sum_{\lla G \rra < q^r} \widetilde{\chi}(t^{M_j + r} + G) = \widetilde{\chi}(t)^{M_j+r} + \sum_{0 \leq l < r} \widetilde{\chi}(t)^l \sum_{\lla G \rra < q^{r-l}} \chi(G),
$$
the double sum on the right-hand side being independent of $j = 1,2$. It follows from this that $\widetilde{\chi}(t)^{M_1} = \widetilde{\chi}(t)^{M_2}$, so choosing e.g., $M_2 = M_1 + 1$ yields the claim $\widetilde{\chi}(t)=1$ in this case.
\end{proof}

\begin{proof}[Proof of Proposition~\ref{prop_lexico} when $P = t$]
Let $\{n_j\}_{1 \leq j \leq k}$ be an increasing sequence of integers satisfying $n_{j+1} > n_{j} + \nu + r$ for each $1 \leq j \leq k-1$. Define
$$
N_j := \langle 1\rangle\sum_{1 \leq i \leq j} q^{n_i}
$$ 
for each $1 \leq j \leq k$; since $\langle \cdot \rangle$ is a bijection on $\mb{F}_q$ and $\langle 0 \rangle = 0$ we note that $N_j > 0$ for each $j \geq 1$. Note also that if $G$ is monic and satisfies $\langle G \rangle < N_j$ then either $\deg{G} < n_j$ or else $G = t^{n_j} + M$, where $\langle M \rangle < N_{j-1}$. In the latter case, since $n_j > n_{j-1} + \nu$ we have $\xi e_{\theta}(t^{n_j}+M) = \xi e_{\theta}(t)^{n_j}$ whenever $\langle M \rangle < N_{j-1}$. Furthermore, if $G \neq 0$ then $\nu_t(t^{n_j} + G) = \nu_t(G)$ and thus by our choice of $n_j$ we have $\widetilde{\chi}(t^{n_j} + G) = \widetilde{\chi}(G)$.  Lemma~\ref{lem:peqtRed} shows that we may assume $\widetilde{\chi}(t) = 1$ and we thus obtain
\begin{align}\label{eq134}
\sum_{\ss{\lla G \rra < N_k \\ G \in \mc{M}}} g(G) &= \sum_{G \in \mc{M}_{<n_k}} g(G) + \xi e_{\theta}(t)^{n_k} \left(\sum_{0 < \lla G \rra < N_{k-1}} \widetilde{\chi}(t^{n_k} + G) + 1\right) \nonumber\\
&=  \xi e_{\theta}(t)^{n_k} S_{N_{k-1}}(\widetilde{\chi}) + O(1).
\end{align}
We similarly have for $1 \leq m \leq k-1$ that
\begin{align*}
S_{N_m}(\widetilde{\chi}) = \sum_{\ss{\deg{G} < N_m}} \widetilde{\chi}(G) + 1 + \sum_{0 < \lla G \rra < N_{m-1}} \widetilde{\chi}(t^{n_m} + G) &= \Sigma_{N_m}(\widetilde{\chi}) + 1 + S_{N_{m-1}}(\widetilde{\chi}),
\end{align*}
and on iterating this we get
\begin{align}\label{eq135}
S_{N_m}(\widetilde{\chi}) = \sum_{1 \leq j \leq m}\left( \Sigma_{N_j}(\widetilde{\chi}) + 1\right) + O(1).    
\end{align}

We  now deduce from~\eqref{eq134},~\eqref{eq135} that
$$
\left|\sum_{1 \leq j \leq k-1} (\Sigma_{N_j}(\widetilde{\chi}) + 1) \right| = O(1).
$$
On the other hand, we have $\Sigma_{N_j}(\widetilde{\chi}) = c_q \sum_{G \in \mc{M}_{<N_j}} \widetilde{\chi}(G)$ (where $c_q$ is given by~\eqref{eqq119}), and by Lemma~\ref{lem_omega1_lsd}(2) the map $n \mapsto \sum_{G \in \mc{M}_{<n}} \widetilde{\chi}(G)$ is constant (since $\deg t = 1$ and $\widetilde{\chi}(t) = 1$). Thus, we in fact obtain that
$$
(k-1)\left|c_q \sum_{G \in \mc{M}_{<n}} \widetilde{\chi}(G) + 1\right| = O(1),
$$
for any $n \geq 1$. Taking $n = 1$, we see that $c_q \sum_{G \in \mc{M}_{<1}}\widetilde{\chi}(G) = c_q \chi(1) \neq -1$ in any case. We obtain the contradiction $k \ll 1$, and the claim is proved.
\end{proof}
We will split the remaining case $P \neq t$ into two subcases depending on whether $\alpha \in \mb{Q}$ or not (recall from~\eqref{eq200} that $e(\alpha)=\widetilde{\chi}(P)$). Our argument in both subcases has a common setup that we introduce presently.

Pick a sequence $(m_k)_{k\geq 1}$ such that $m_k - m_{k-1} \geq 10 \nu$, and let $a \geq 1$ be an integer to be chosen later, which is bounded in terms of $\alpha$ and $P^r$. 
Let $(n_k)_{k\geq 1}=(m_k \phi(P^r)+a)_{k\geq 1}$. We assume furthermore that $m_k$ is chosen so that $m_k \geq 2n_{k-1}$, so e.g., $m_k/m_{k-1} \geq 10 \nu \phi(P^r)$ is sufficient.
As in the case $P = t$, define a sequence $(N_k)_{k\geq 1}$ by
\begin{align}\label{eq:Nkdef}
N_k= \langle 1\rangle \sum_{j=1}^k q^{n_j}.   
\end{align}

Note that, by Euler's theorem over $\mathbb{F}_q[t]$, we have 
\begin{align*}
t^{\phi(P^r)}\equiv 1\pmod{P^r}.     
\end{align*}
This means that 
\begin{align} \label{eq:EulerExp}
t^{\phi(P^r)}= 1 + P^v G_0 
\end{align}
where $v \geq r$ and $G_0\in \mc{M}$ is coprime to $P$. If $m_k$ is chosen to be a power of $q$ then by the binomial formula,
\begin{equation}\label{eqq116}
t^{n_k-a} = (1+P^vG_0)^{m_k} = 1 + (P^vG_0)^{m_k}.
\end{equation}
Since, by  assumption, $m_k \geq 2n_{k-1}$ we obtain that 
\begin{align}\label{eqq109}
t^{n_k} \equiv t^a \pmod{P^{2rn_{k-1}}}. 
\end{align}
The fact that $t^{n_k-a}-1$ is highly divisible by $P$ will be used crucially in the sequel. 
We now split our sum, similarly as in~\eqref{eq134}, as
\begin{align*}
S_{N_k}^{\mc{M}}(g)&=\sum_{\substack{G\in \mathcal{M}\\\langle G\rangle< q^{n_k}}}g(G)+\sum_{\substack{G\in \mathbb{F}_q[t]\\0\leq \langle G\rangle< N_{k-1}}}g(t^{n_k}+G)\\
&=\Sigma_{n_k}^{\mathcal{M}}(g)+\sum_{\substack{G\in \mathbb{F}_q[t]\\0\leq \langle G\rangle< N_{k-1}}}g(t^{n_k}+G).    
\end{align*}

Note that by~\eqref{eqq109} and the fact that $n_k>n_{k-1}+\nu$, we have
\begin{align*}
g(t^{n_k}+G)=e(\theta n_k)\xi(t)^{n_k}\widetilde{\chi}(t^{a}+G)    
\end{align*}
for all $G\in \mc{M}_{\leq n_{k-1}}\setminus \{-t^a\}$. Also note that the conditions $\langle G+t^a\rangle< n$ and $\langle G\rangle< n$ are equivalent whenever $q^{a+1}\mid n$, and $q^{a+1}\mid N_j$ for all $j\geq 1$. Hence, we obtain
\begin{align}\label{eqq110}\begin{split}
S_{N_k}^{\mc{M}}(g)&=\Sigma_{n_k}^{\mc{M}}(g)+e(\theta n_k)\xi(t)^{n_k}S_{N_{k-1}}(\widetilde{\chi})+g(t^{n_k}-t^a)
\end{split}
\end{align}
Similarly, for all $k\geq 1$,  we have
\begin{align}\label{eqq111}\begin{split}
S_{N_{k}}(\widetilde{\chi})&=\Sigma_{n_{k}}(\widetilde{\chi})+\sum_{\substack{G\in \mathbb{F}_q[t]\\0 \leq \langle G\rangle< N_{k-1}}}\widetilde{\chi}(t^{n_k}+G)\\
&= \Sigma_{n_{k}}(\widetilde{\chi})+S_{N_{k-1}}(\widetilde{\chi})+\widetilde{\chi}(t^{n_{k}}-t^a), \end{split}
\end{align}
where $N_0:=0$. Iterating~\eqref{eqq111} and substituting into~\eqref{eqq110} 
produces
\begin{align}\label{eqq117}\begin{split}
S_{N_k}^{\mathcal{M}}(g)&=e(\theta n_k)\xi(t)^{n_k}\left(\sum_{j=1}^{k-1}\Sigma_{n_{j}}(\widetilde{\chi})+\sum_{j=1}^{k}\widetilde{\chi}(t^{n_j}-t^a)\right)+\Sigma_{n_k}^{\mathcal{M}}(g)+O(1)\\
&=e(\theta n_k)\xi(t)^{n_k}\left(\sum_{j=1}^{k-1}\Sigma_{n_{j}}(\widetilde{\chi})+\sum_{j=1}^{k}\widetilde{\chi}(t^{n_j}-t^a)\right)+O(1),
\end{split}
\end{align}
where we used the assumption $\Sigma_{n_k}^{\mathcal{M}}(g)= O(1)$. This leads to
\begin{align}\label{eqq117b}
\sum_{j=1}^{k-1}\Sigma_{n_{j}}(\widetilde{\chi})+\sum_{j=1}^{k}\widetilde{\chi}(t^{n _j}-t^a)=O(1).
\end{align}
At this point, we may distinguish between the remaining two cases.

\begin{proof}[Proof of Proposition~\ref{prop_lexico} when $P \neq t$] As mentioned, the proof splits into two subcases. 

\subsubsection{\textbf{\textrm{Case 1:}} $P \neq t$, $\alpha \notin \mb{Q}$} Let $G_0$ and $v \geq r$ be as in~\eqref{eq:EulerExp}. Let $d := \text{ord}(\chi(G_0))$ and let $\beta$ be a limit point of the sequence $\{vq^{An}\alpha \pmod{1}\}_{n \geq 1}$, where $A = 20C\nu r \deg{P}$ and $C \geq 1$ is a large integer depending only on $\alpha$ to be chosen below. By the pigeonhole principle we may select $\ell_1 < \cdots <\ell_k$ sufficiently large in terms of $\alpha$ and $\deg{P}$ such that 
$$
\|vq^{A\ell_j}\alpha - \beta\| < \tfrac 1{100},
$$
and so that $q^{A\ell_j} \equiv c_0 \pmod{d}$ for all $1 \leq j \leq k$ and some $1 \leq c_0 \leq d$.
We now set $m_j  = q^{A \ell_j}$, and $a = \gamma \deg{P}$, where $1 \leq \gamma = \gamma(\alpha) \leq C$ is an integer to be chosen later. With this choice, we have $n_j = \phi(P^r) q^{A\ell_j} + \gamma \deg{P}$, and for suitably large $\ell_1$ we have $n_1 \geq 10 C \deg{P} + a \geq 2a$. For $k \geq 1$ we may verify the required inequalities $n_k-n_{k-1} > \nu+r$ and 
$$
m_k \geq 10\phi(P^r) m_{k-1} = 10 (n_{k-1}-a) \geq 5  n_{k-1} \text{ for all } k \geq 1.
$$

By Lemma~\ref{lem_omega1_lsd}(1), for any $j\geq 1$ we have
\begin{align*}
\Sigma_{n_j}^{\mc{M}}(\widetilde{\chi}) &= \sum_{\substack{M' \in \mc{M} \\ \deg{M'} < \nu + r\deg{P}}} \chi(M') \sum_{0 \leq \ell < (n_j-\deg{M'})/\deg{P}} e(\alpha \ell) \\
&= \frac{1}{1-e(\alpha)} \sum_{\substack{M' \in \mc{M} \\ \deg{M'} < \nu + r\deg{P}}} \chi(M') \left(1- e\left(\alpha \left(1 + \llf \frac{n_j-\deg{M'}}{\deg{P}}\rrf\right)\right)\right).
\end{align*}
As the sum over $M'$ (without the bracketed expression) vanishes, we may ignore the term $1$ in the brackets. 
Recalling~\eqref{eqq118},~\eqref{eqq119} and that $\deg{P}|a$, we can rewrite $\Sigma_{n_j}(\widetilde{\chi})$ as
$$
\Sigma_{n_j}(\widetilde{\chi}) = \frac{- e(\alpha(1+\gamma))c_q}{1-e(\alpha)} \sum_{\substack{M' \in \mc{M} \\ \deg{M'} < \nu + r\deg{P}}} \chi (M') e\left(\llf \frac{m_j\phi(P^r) - \deg{M'}}{\deg{P}}\rrf \alpha \right) .
$$
Summing over $j$, we obtain
\begin{align*}
&\sum_{1 \leq j \leq k-1} \Sigma_{n_j}(\widetilde{\chi}) \\
&= e(\gamma \alpha) \cdot \Bigg( -\frac{c_qe(\alpha)}{1-e(\alpha)} \sum_{\substack{M' \in \mc{M} \\ \deg{M'} < \nu + r\deg{P}}} \chi(M') \sum_{1 \leq j \leq k-1} e\left(\llf \frac{m_j\phi(P^r) -\deg{M'}}{\deg{P}}\rrf \alpha\right)\Bigg) \\
&=:  e(\gamma \alpha) \mc{S}(\alpha).
\end{align*}
Note that $\mc{S}(\alpha)$ is independent of $\gamma$. Splitting off $\widetilde{\chi}(t)^a$ in~\eqref{eqq117b}, that expression becomes
\[
e(\gamma \alpha) \mc{S}(\alpha) + \chi(t^{\deg{P}})^{\gamma} \sum_{j=1}^{k}\widetilde{\chi}(t^{n _j-a}-1)=O(1).
\]
Now, since $\alpha \notin \mb{Q}$ and $\chi(t^{\deg{P}})$ is a root of unity of order $d$, it follows that (taking $C = C(\alpha)$ large enough ) an integer $\gamma = \gamma(\alpha) \in [1,C]$ can be chosen so that
\[
|\text{arg}(e(\gamma \alpha) \mc{S}(\alpha)) - \text{arg}(\chi(t^{\deg{P}})^{\gamma} \sum_{j=1}^{k}\widetilde{\chi}(t^{n _j-a}-1))| \in (-\tfrac 1{100}, \tfrac 1{100}).
\]
Hence,~\eqref{eqq117b} in fact implies that
\begin{align}\label{eq:redToJustchi}
\chi(t^{\deg{P}})^{\gamma}\sum_{j=1}^{k}\widetilde{\chi}(t^{n _j-a}-1) = O(1).
\end{align}
Now by construction, for each $1 \leq j \leq k$,
$$
\widetilde{\chi}(t^{n_j-a}-1) = \widetilde{\chi}(P)^{vq^{A\ell_j}}\chi(G_0)^{q^{A\ell_j}} =  e(vq^{A\ell_j}\alpha) \chi(G_0)^{c_0},
$$
and by choice of $\ell_j$ we have that $|e(vq^{A\ell_j}\alpha)-e(\beta)| \leq \frac{2\pi}{100} < \tfrac 1{10}$. It follows that
$$
\left|\sum_{j = 1}^k \widetilde{\chi}(t^{n_j-a}-1)\right| = \left|\chi(G_0)^{c_0} e(\beta) k + \sum_{j = 1}^k \chi(G_0)^{c_0}(e(vq^{A\ell_j}\alpha)-e(\beta))\right|  \geq \tfrac{9 k}{10},
$$
which contradicts~\eqref{eq:redToJustchi} for $k$ sufficiently large.
This completes the proof in Case 1.

\subsubsection{\textbf{Case 2.} $P \neq t$ and $\alpha \in \mb{Q}$} In this case, we may find $b \geq 1$ such that $\widetilde{\chi}(P) = e(\alpha)$ is a $b$th root of unity. We select $m_k = q^{A\ell_k}$, where $A = 2\nu r\deg{P}$ and $\ell_k$ is chosen so that $\ell_k-\ell_{k-1} \geq 10$ and also so that $\widetilde{\chi}(t^{\phi(P^r)q^{A\ell_k}}-1) = g_0$ for all $k \geq 1$, where $g_0 \in S^1$ (for this, it suffices for $q^{A\ell_k}$ to be constant modulo $[b,d]$, where as above $d = \text{ord}(\chi(G_0))$. We also pick $a \in [1,b \deg{P}]$.

By~\eqref{eqq116} we have
\begin{align}\label{eqq112}
\widetilde{\chi}(t^{n_k}-t^a)=\widetilde{\chi}(t)^a\widetilde{\chi}(t^{n_k-a}-1)=\chi(t)^ag_0,
\end{align}
by the definition of $g_0$.

We combine~\eqref{eqq117},~\eqref{eqq118} and~\eqref{eqq119}, using the fact (following from Lemma~\ref{lem_omega1_lsd}(2) and the assumption $\widetilde{\chi}(P)^b=1$) that $n\mapsto \Sigma^{\mathcal{M}}_n(\widetilde{\chi})$ is $b\deg{P}$-periodic and $n_j\equiv a\pmod {b\deg{P}}$. We then see that
\begin{align}\label{eq131}
\sum_{j=1}^{k-1}\Sigma_{n_{j}}(\widetilde{\chi})+\sum_{j=1}^{k}\widetilde{\chi}(t^{n _j}-t^a)&=k\left(c_q\Sigma_{a}^{\mathcal{M}}(\widetilde{\chi})+g_0\chi(t)^a\right)+O(1).
\end{align}
Now, as the left-hand side of~\eqref{eq131} is $O(1)$, we must have 
\begin{align}\label{eq127}
 c_q\Sigma_{a}^{\mathcal{M}}(\widetilde{\chi})+g_0\chi(t)^a=0
\end{align}
for all $a\in [1,b\deg{P}]$. We have $c_q=0$ or $c_q=q-1$, and the first of these is clearly impossible, since $|g_0\widetilde{\chi}(t)^a|=1$. Hence,~\eqref{eq127} becomes
\begin{align}\label{eq127b}
 (q-1)\Sigma_{a}^{\mathcal{M}}(\widetilde{\chi})+g_0\chi(t)^a=0
\end{align}
Since the sequence $n\mapsto \Sigma_{n}^{\mathcal{M}}(\widetilde{\chi})$ is $b\deg{P}$-periodic, and $n\mapsto \chi(t)^n$ is also $b\deg{P}$-periodic, we deduce that 
\begin{align}\label{eq128}
(q-1)\Sigma_{n}^{\mathcal{M}}(\widetilde{\chi})+g_0\chi(t)^n=0  \end{align}
for all $n\geq 0$. 

From~\eqref{eq128} we see that, for $\text{Re}(s) > 1$, we have
\begin{align*}
\sum_{G\in \mathcal{M}}\widetilde{\chi}(G)q^{-s\deg{G}}=\sum_{n\geq 0}\Sigma_{n}^{\mathcal{M}}(\widetilde{\chi})q^{-sn}=-\frac{g_0(q-1)^{-1}}{1-\chi(t)q^{-s}}.    
\end{align*}
But on the other hand in the same region of $s$ by~\eqref{eqq19} we have 
\begin{align*}
\sum_{G\in \mathcal{M}}\widetilde{\chi}(G)q^{-s\deg{G}}=\prod_{R\in \mathcal{P}}\sum_{k\geq 0}\widetilde{\chi}(P)^kq^{-sk\deg{R}}=\frac{L(s,\chi)}{1-\widetilde{\chi}(P)q^{-s\deg{P^r}}}.    
\end{align*}
Comparing these, we see that
\begin{align}\label{eq129}
 L(s,\chi)=-\frac{g_0(q-1)^{-1}(1-\widetilde{\chi}(P)q^{-s\deg{P^r}})}{1-\chi(t)q^{-s}};   
\end{align}
initially this holds for $\text{Re}(s) > 1$, but by analytic continuation we in fact have this for all $s$. In particular, if $\deg{P^r}\geq 2$, this implies that $L(s,\chi)$ has a root other than $s = 0$ off the critical line $\text{Re}(s) = 1/2$. But by GRH over function fields this is not possible. Hence, $\deg{P^r} = r\deg{P} =1$, and $r = 1$. Then, $P$ being monic and coprime to $t$ implies that $P=t+c$ for some $c\in \mathbb{F}_q\setminus \{0\}$. As $c_q \neq 0$ it follows that $\chi$ is 1 on $\mb{F}_q^{\times}$, and therefore $\chi(t) = \chi(-c) = 1$. Since $L(s,\chi)$ is analytic, it follows that $\widetilde{\chi}(P) = 1$ as well, but then $L(s,\chi) = -g_0(q-1)^{-1} \neq 0$ for all $s$. On the other hand, since $\deg{P} = 1$ we obtain $\sum_{M \in \mc{M}_n} \chi(G) = 0$ for all $n \geq 1$, and thus
$$
L(s,\chi) = \sum_{M \in \mc{M}} \chi(M)q^{-s\deg{M}} = 1.
$$

Comparing with~\eqref{eq129}, we obtain $(q-1)^{-1}=-g_0$, so that as $g_0\in S^1$ we must have $q=2$ and $g_0=-1$. Hence, $\widetilde{\chi}$ must be a generalized character $\pmod{t+1}$, and additionally $\widetilde{\chi}(t+1)=1$. But now if $G\in \mathbb{F}_2[t]$ is any monic polynomial, then by changing bases we can write
$$
G = \sum_{0 \leq j \leq r} a_j(t+1)^j,
$$
with $a_j \in \{0,1\}$. If $j_0$ is the minimal index for which $a_j \neq 0$ then we immediately find that $\widetilde{\chi}(G) = \widetilde{\chi}(a_{j_0}) = 1$. Hence $\widetilde{\chi} \equiv 1$. But this contradicts the assumption $|\Sigma_a^{\mc{M}}(\widetilde{\chi})| = |-g_0\widetilde{\chi}(t)^a| = 1$ for all $a \geq 1$, since
$$
\Sigma_1^{\mc{M}}(\widetilde{\chi}) = 1+\widetilde{\chi}(t) + \widetilde{\chi}(t+1) = 3.
$$
This completes the analysis of Case 2 and the proof of Proposition~\ref{prop_lexico} in this case. 
\end{proof}

\begin{rem}
We remark that the proof of Proposition~\ref{prop_lexico} gives at best a growth rate of $\gg \log \log N$ for the lexicographic discrepancy $S_{N}^{\mathcal{M}}(g)$. This is because the sequence $(n_k)_k$ must satisfy $n_k\geq 2n_{k-1}$ so that by~\eqref{eq:Nkdef} we have $k\ll \log \log N_k$. 
\end{rem}

The proof of Theorem~\ref{thm_lexicographic} is now complete.

\section{The Long Sum Discrepancy}\label{sec_EDP3}

We will next prove our characterization result for unboundedness of the long sum discrepancy (Theorem~\ref{thm_lsdisc_genchar}), as well as Proposition~\ref{prop_bddlf_unbddsf} that complements it. We begin the proof of Theorem~\ref{thm_lsdisc_genchar} with the following simple observation.

\begin{lem}\label{lem_aux}
Let $f: \mc{M} \to \mb{U}$ be a modified character associated with a non-principal Dirichlet character $\chi$ modulo $Q\in \mc{M}$, and let $N > \deg{Q}$ be large. Then 
\begin{align*}
\sum_{G \in \mc{M}_{\leq N}} f(G) = \sum_{0 \leq m < \deg{Q}} \left(\sum_{\substack{A \pmod{Q} \\ A \in \mc{M}_{\leq m}}} \chi(A)\right) \sum_{\substack{\textnormal{rad}(D)|Q \\ \deg{D} = N-m}} f(D).
\end{align*}
\end{lem}
\begin{proof}
We split the sum on the left-hand side according to the common factors of $G$ with $Q$ to obtain
\begin{align*}
\sum_{G \in \mc{M}_{\leq N}} f(G) = \sum_{\textnormal{rad}(D)|Q} f(D)\sum_{G' \in \mc{M}_{\leq N-\deg{D}}} \chi(G') = \sum_{A\pmod Q} \chi(A) \sum_{\textnormal{rad}(D)|Q} f(D) \sum_{\substack{G' \in \mc{M}_{\leq N-\deg{D}} \\ G' \equiv A \pmod{Q}}} 1.
\end{align*}
We separate the contribution with $\deg{D} \leq N-\deg{Q}$ from its complement. Observe that when $\deg{D} \leq N-\deg{Q}$, the inner sum above is independent of $A$. Thus, orthogonality implies that its contribution is 0. On the other hand, if $\deg{D} > N-\deg{Q}$ and $G' \in \mc{M}_{\leq N-\deg{D}}$ with $G' \equiv A \pmod{Q}$ then $G' = A$. It follows that
\begin{align*}
\sum_{G \in \mc{M}_{\leq N}} f(G) = \sum_{\substack{\textnormal{rad}(D)|Q \\ N-\deg{Q} < \deg{D} \leq N}} f(D) \sum_{\substack{A \pmod{Q} \\ A \in \mc{M}_{\leq N-\deg{D}}}} \chi(A).
\end{align*}
Splitting the sum according to the size of $\deg{D}$, we get
\begin{align*}
\sum_{G \in \mc{M}_{\leq N}} f(G) = \sum_{0 \leq m < \deg{Q}} \left(\sum_{\substack{A \pmod{Q} \\ A \in \mc{M}_{\leq m}}} \chi(A)\right) \sum_{\substack{\textnormal{rad}(D)|Q \\ \deg{D} = N-m}} f(D),
\end{align*}
as claimed.
\end{proof}

\begin{proof}[Proof of Theorem~\ref{thm_lsdisc_genchar}]
Let $N > \deg{Q}$ be large. By the residue theorem, we have
\begin{align*}
\sum_{\substack{\textnormal{rad}(D)|Q \\ \deg{D} = N-m}} f(D) &= \frac{1}{2\pi i} \int_{|z| = r} \left(\sum_{\textnormal{rad}(D)|Q} f(D)z^{\deg{D}}\right) \frac{dz}{z^{N-m+1}}\\
&= \frac{1}{2\pi i } \int_{|z| = r}\prod_{P|Q} \left(1-f(P)z^{\deg{P}}\right)^{-1} z^m \frac{dz}{z^{N+1}},
\end{align*}
for any $r \in (0,1)$. 
Using Lemma~\ref{lem_aux} together with this expression for each $m \leq \deg{Q}-1$, we have
\begin{align}
&\sum_{G \in \mc{M}_{\leq N}} f(G) = \sum_{0 \leq m < \deg{Q}} \left(\sum_{\substack{A \pmod{Q} \\ A \in \mc{M}_{\leq m}}} \chi(A)\right) \sum_{\substack{\textnormal{rad}(D)|Q \\ \deg{D} = N-m}} f(D) \nonumber\\
&= \frac{1}{2\pi i } \int_{|z| = r} \prod_{P|Q}\left(1-f(P)z^{\deg{P}}\right)^{-1} \left(\sum_{A \in \mc{M}_{<\deg{Q}}} \chi(A) \sum_{\deg{A} \leq m \leq \deg{Q}-1} z^m \right) \frac{dz}{z^{N+1}} \nonumber\\
&= \frac{1}{2\pi i} \int_{|z| = r} \prod_{P|Q}\left(1-f(P)z^{\deg{P}}\right)^{-1} (1-z)^{-1}\sum_{A \in \mc{M}_{<\deg{Q}}} \chi(A) z^{\deg{A}} (1-z^{\deg{Q}-\deg{A}}) \frac{dz}{z^{N+1}} \nonumber,
\end{align}
where we used the geometric sum formula in the last step. By the orthogonality of characters, we have $\sum_{A\in \mathcal{M}_{<\deg{Q}}}\chi(A)z^{\deg{Q}}=0$. Thus, if $z = q^s$ for some $s \in \mb{C}$ then by writing $\mc{L}(z,\chi) := L(s,\chi)/(1-z)$
(see~\eqref{eqq19} for the definition of $L(s,\chi)$), 
the previous expression simplifies to
\begin{align}
\frac{1}{2\pi i} \int_{|z| = r} \mc{L}(z,\chi) \prod_{P|Q} \left(1-f(P)z^{\deg{P}}\right)^{-1}  \frac{dz}{z^{N+1}}, 
\end{align}
using orthogonality in the last step.

Now let $\lambda_1,\ldots,\lambda_J$ be the collection of distinct roots of $\prod_{P|Q}(1-f(P)z^{\deg{P}})$, with respective multiplicities satisfying $b_1\leq \ldots \leq b_J$; note that $\lambda_j \in S^1$ for all $j$. A partial fraction decomposition of the reciprocal of this polynomial yields coefficients $\{a_{j,l}\}_{1 \leq l \leq b_j, 1 \leq j \leq J}$ such that
\begin{align*}
\prod_{P|Q} (1-f(P)z^{\deg{P}})^{-1} = \prod_{1 \leq j \leq J} (1-\lambda_j z)^{-b_j} = \sum_{1 \leq j \leq J} \sum_{1 \leq l \leq b_j} \frac{a_{j,l}}{(1-\lambda_jz)^l}.
\end{align*}
Noting that for each pair $(j,l)$ we have the formal power series expansion
$$(1-\lambda_j z)^{-l} = \sum_{k \geq 0} \binom{l-1+k}{k}\lambda_j^k z^k,$$
we see that 
\begin{align}
\sum_{G \in \mc{M}_{\leq N}} f(G) &= \sum_{k \geq 0} \sum_{1 \leq j \leq J} \sum_{1 \leq l \leq b_j} a_{j,l}\binom{l-1+k}{k} \lambda_j^k \left(\frac{1}{2\pi i } \int_{|z| = r} \mc{L}(z,\chi) z^k \frac{dz}{z^{N+1}}\right) \nonumber\\
&= \sum_{1 \leq j \leq J} \sum_{1 \leq l \leq b_j} a_{j,l} \sum_{N-\deg{Q} < k \leq N} \left([z^{N-k}] \mc{L}(z,\chi)\right) \binom{l-1+k}{k}\lambda_j^k, \label{eq_forab}
\end{align}
where, given a formal power series $F(z)$ in $z$ we write $[z^m]F(z)$ to denote the $m$th coefficient of $F$, for $m \in \mb{N}\cup \{0\}$. We shall use this last expression to prove both parts of the proposition, beginning with part b). \\
\textbf{Part b).} By hypothesis, $b = b_J\geq 2$. Let $1 \leq i \leq J$ be the minimal index for which $b_i = b_{i+1} = \cdots = b_J\leq \deg{Q}$. As $\binom{l-1+k}{k} = \frac{N^{l-1}}{(l-1)!} + O_{\deg(Q)}(N^{l-2})$ for any $l \geq 2$ and $k \in (N-\deg{Q},N]$, we get
\begin{align*}
&\sum_{1 \leq j \leq J} \sum_{1 \leq l \leq b_j} a_{j,l} \sum_{N-\deg{Q}< k \leq N} \left([z^{N-k}]\mc{L}(z,\chi)\right) \binom{l-1+k}{k} \lambda_j^k \\
&= \sum_{i \leq j \leq J} a_{j,b} \sum_{N-\deg{Q}< k \leq N} \binom{b-1+k}{k} \left([z^{N-k}]\mc{L}(z,\chi)\right) \lambda_j^k + O_{\deg(Q)}(N^{b-2}) \\
&= \frac{N^{b-1}}{(b-1)!} \sum_{i \leq j \leq J} a_{j,b} \sum_{0 \leq m <\deg{Q}} \left([z^m]\mc{L}(z,\chi)\right) \lambda_j^{N-m}  + O_{\deg(Q)}(N^{b-2}) \\
&= \frac{N^{b-1}}{(b-1)!} \sum_{i \leq j \leq J} a_{j,b} \lambda_j^N \mc{L}(\bar{\lambda_j},\chi) + O_{\deg(Q)}(N^{b-2}),
\end{align*}
where in the last step we made the change of variables $m=N-k$, which leads to the power series in the penultimate line simplifying to $\lambda_j^{N}\mathcal{L}(\bar{\lambda_j},\chi)$.

We note that $\mc{L}(\bar{\lambda_j},\chi) \neq 0$ for all $j$ because by GRH~\cite[Thm. 5.5 and Ex. 5.2.2]{EffHay} we know that $\mc{L}(z,\chi)$ has no zeros off the circle $|z| = q^{-1/2}$, aside from a simple zero at $z = 1$ (which has been cancelled in the definition of $\mc{L}(z,\chi)$). Moreover, $a_{j,b} \neq 0$ for all $i \leq j \leq b$ as well, otherwise the maximal power of $(1-\lambda_jz)^{-1}$ in the partial fraction decomposition would be strictly smaller than $b$.  Finally, applying Dirichlet's theorem we can find a sequence of $\{N_r\}_r$ such that $\max_{i \leq m \leq J} |\lambda_m^{N_r}-1| \leq \e$ for any specific choice of $\e > 0$ (chosen small relative to $Q$ and $J$). It follows that for all $l \in \{0,\ldots,J-i-1\}$ we have
\begin{align}\label{eqq20}
\sum_{i \leq j \leq J} a_{j,b} \mc{L}(\bar{\lambda_j},\chi) \lambda_j^{N_r + l}  = \sum_{i \leq j \leq J} a_{j,b} \mc{L}(\bar{\lambda_j},\chi) \lambda_j^l + O_J(\e),
\end{align}
and thanks to the invertibility of the van der Monde matrix generated by $\lambda_i,\ldots,\lambda_J$ (which are distinct by assumption) the expression~\eqref{eqq20} is $\neq 0$ for at least one $l$ and some $\e>0$ sufficiently small. This implies then that
\begin{align*}
\max_{N_r \leq N \leq N_r+J} \left|\sum_{G \in \mc{M}_{\leq N}} f(G)\right| \asymp_{Q} N_r^{b-1},
\end{align*}
as $r \to \infty$. This completes the proof of part b). 

\textbf{Part a).} From~\eqref{eq_forab}, we have
\begin{align*}
\sum_{G \in \mc{M}_{\leq N}} f(G) = \sum_{1 \leq j \leq J} \sum_{1 \leq l \leq b_j} a_{j,l} \sum_{N-\deg{Q} < k \leq N} ([z^{N-k}]\mc{L}(z,\chi)) \binom{l-1+k}{k}\lambda_j^k.
\end{align*}
Since $b_j \leq b_J$, we have $b_j = 1$ for all $j$. As above, we obtain
\begin{align}\label{eqq23}
\sum_{G \in \mc{M}_{\leq N}} f(G) = \sum_{1 \leq j \leq J} a_{j,1} \sum_{N - \deg{Q}<k\leq N} ([z^{N-k}] \mc{L}(z,\chi)) \lambda_j^k = \sum_{1 \leq j \leq J} a_{j,1} \lambda_j^N\mc{L}(\bar{\lambda_j},\chi).
\end{align}
Note that $\lambda_j \in S^1$ for all $j$, and $\mathcal{L}(z,\chi)$, is holomorphic and thus bounded on $S^1$ (in terms solely of the conductor $Q$). Furthermore, $a_{j,1}$ depends only on $Q$. It follows that the sum here is $O_{Q}(1)$. This completes the proof.
\end{proof}

This gives the following list of corollaries, which includes Corollary~\ref{cor_lsdisc1}.

\begin{cor}\label{cor_lsdisc}
Let $f : \mc{M} \to S^1$ be a modified character associated with a non-principal character of modulus $Q$. \\
a) If $Q = P^k$ is a prime power then $\mc{D}_{f} < \infty$. \\
b) If $\omega(Q) \geq 2$ and there exist prime divisors $P_1,P_2$ of $Q$ satisfying $\deg{P_1} = \deg{P_2}$ and $f(P_1) = f(P_2)$ then $\mc{D}_f = \infty$. \\
c) If $\omega(Q) \geq 2$ and here exist prime divisors  $P_1,P_2$ of $Q$ satisfying $f(P_1) = f(P_2) = 1$ then $\mc{D}_f = \infty$. \\
d) Suppose $f$ takes values in $\{-1,+1\}$. \\
i) If $\omega(Q) \geq 4$ then $\mc{D}_f = \infty$. \\
ii) If $\omega(Q) = 3$ then $\mc{D}_f < \infty$ if and only if  (up to permutation) the primes $P_1,P_2,P_3$ dividing $Q$ satisfy $f(P_1) = f(P_2) = -1$, $f(P_3) = 1$, and $v_2(\deg{P_1}) \neq v_2(\deg{P_2})$ and $v_2(\deg{P_j}) \geq v_2(\deg{P_3})$ for $j = 1,2$.\\
iii) If $\omega(Q) = 2$ then $\mc{D}_f < \infty$ if and only if  (up to permutation) the primes $P_1,P_2$ dividing $Q$ satisfy $f(P_1) =-1,$ $f(P_2)=1$, and $v_2(\deg{P_1}) \geq v_2(\deg{P_2})$.
\end{cor}

\begin{proof}
a) Since the zeros of the equation $z^m = a$ (with $m = \deg{P}$ and $a = \bar{f(P)}$) are all distinct, Theorem~\ref{thm_lsdisc_genchar} a) implies that the discrepancy is bounded. 

b) Since the expressions $z^{\deg{P_j}}f(P_j) = 1$ are identical for $j = 1,2$, they thus yield identical roots, so by Theorem~\ref{thm_lsdisc_genchar} b) the claim follows. 

c) This follows from Theorem~\ref{thm_lsdisc_genchar} b), as $z^{\deg{P_1}} = 1$ and $z^{\deg{P_2}} = 1$ must share the common root $z = 1$. 

d) i) Let $f:\mc{M}\to \{-1,+1\}$, and let $P_1,P_2,P_3,P_4$ be distinct prime divisors of $Q$. At least two prime divisors $P_1,P_2$ are such that $f(P_1) = f(P_2)$. If the parities of $\deg{P_1}$ and $\deg{P_2}$ are the same then as in the proof of b) the equations $z^{\deg{P_1}} = f(P_1)$ and $z^{\deg{P_2}} = f(P_2)$ will share a common root. Now by c), if the common value of $f(P_1)$ and $f(P_2)$ is 1 then this is true regardless of these parities. Thus, we may assume that the value $1$ occurs at most once among the values $f(P_j)$, for the 4 prime factors of $Q$. But then at least 3 of the primes $P_j$ are such that $f(P_j) = -1$, and among their degrees at least two have the same parity. Thus, we may conclude that $\prod_{P|Q} (1-f(P_j)z^{\deg{P_j}})$ has a multiple root, and the first claim follows from Proposition~\ref{cor_lsdisc} b). 

ii) Now let $P_1,P_2,P_3$ be the prime divisors of $Q$. The argument in i) shows that if at least two of $f(P_i)$ equal to $1$, or all of the $f(P_i)$ equal to $-1$ then the discrepancy is unbounded. We are left with the case where exactly two of the $f(P_i)$ are $-1$; say $f(P_1)=f(P_2)=-1$ and $f(P_3)=1$. One easily sees that 
\begin{align}\label{eqq21}\begin{split}
 &\{z\in \mathbb{C}:\,\, z^m= -1\}\cap  \{z\in \mathbb{C}:\,\, z^n= -1\}\neq \emptyset\quad \textnormal{if and only if}\quad v_2(m)=v_2(n),\\
 &\{z\in \mathbb{C}:\,\, z^m= -1\}\cap  \{z\in \mathbb{C}:\,\, z^n= 1\}\neq \emptyset\quad \textnormal{if and only if}\quad v_2(m)<v_2(n).
 \end{split}
\end{align}
Applying this with $m,n\in \{\deg{P_1},\deg{P_2},\deg{P_3}\}$ yields the claim.

iii) The proof of the case $\omega(Q) = 2$ is almost identical to that of case $\omega(Q)=3$; again one makes use of~\eqref{eqq21}.
\end{proof}

\begin{proof}[Proof of Proposition~\ref{prop_bddlf_unbddsf}]
This follows by generalizing the Polymath 5 example in~\cite{polymath_example} of a completely multiplicative function having bounded long sum discrepancy.

For $d \geq 1$ define the quantities
\begin{align*}
\alpha_d=\sum_{G\in \mc{M}_d}\Lambda(G)f(G),\quad \beta_d=\sum_{G\in \mc{M}_d}f(G).    
\end{align*}
Using $\deg{G}=\sum_{D\mid G}\Lambda(D)$ and the complete multiplicativity of $f$, we obtain the recursion
\begin{align}\label{recur}
d\beta_d=\sum_{i=1}^d \alpha_i\beta_{d-i}.    
\end{align}
It was shown by Polymath 5~\cite{polymath_example} that there exist a constant $C$ and a completely multiplicative function $f:\mathcal{M}\to \{-1,+1\}$ for which the corresponding $\alpha_i$ satisfy $|\alpha_i|<q^i$ for all $i\geq C$ and for which $0\leq \sum_{0\leq i\leq d}\beta_i\leq C$ for all $i\geq 1$ (Polymath 5 stated their result in the form that if the size $q$ of the field is large enough, then $\sum_{0\leq i\leq d}\beta_i\in \{0,1\}$ for all $i$, but the same proof gives the claim above for all $q$.).  Since the $\beta_i$ are completely determined by the $\alpha_i$, this then means that $\mathcal{D}_g\leq C$ for any completely multiplicative $g$ that produces the same sequence of $\alpha_i$. From this we deduce that there are uncountably many choices of $g$: for each subset $S$ of $\mathbb{N}\cap [C+1,\infty)$, we may form a new completely multiplicative function $f_S$ which is obtained from $f$ by choosing for each $d\in S$ two irreducibles $P_{1,d}, P_{2,d}$ of degree $d$ with $f(P_{1,d})=-f(P_{2,d})$, putting $f_S(P_{j,d}) = -f(P_{j,d})$ for $j = 1,2$, and setting $f_S(P)=f(P)$ at all other irreducibles $P$. The new function $f_S$ has the same sequence of $\alpha_i$ associated with it as to $f$, so it too has discrepancy bounded by $C$.
\end{proof}

\bibliography{FFBib}
\bibliographystyle{plain}

\end{document}